\documentclass{amsart}

\usepackage{amsmath}
\usepackage{amssymb}
\usepackage{MnSymbol}
\usepackage{bbm}
\usepackage{epsfig}  		
\usepackage{xypic}
\usepackage{hyperref}
\usepackage{mathrsfs}
\usepackage{lineno}
\usepackage{xcolor}

\newtheorem{mthm}{Theorem}
\renewcommand\themthm{\Alph{mthm}}
\newtheorem{mcor}[mthm]{Corollary}

\newtheorem{thm}{Theorem}[section]
\newtheorem{prop}[thm]{Proposition}
\newtheorem{lem}[thm]{Lemma}
\newtheorem{cor}[thm]{Corollary}

\newtheorem*{thm*}{Theorem}

\theoremstyle{definition}
\newtheorem{definition}[thm]{Definition}
\newtheorem{example}[thm]{Example}
\newtheorem*{example*}{Example}

\theoremstyle{remark}

\newtheorem*{remark}{Remark}

\numberwithin{equation}{section}

\DeclareMathAlphabet{\mathpzc}{OT1}{pzc}{m}{it}
\renewcommand{\mathcal}[1]{\text{$\mathpzc{#1}$}}

\newcommand{\CC}{\mathbbm{C}}
\newcommand{\RR}{\mathbbm{R}}

\newcommand{\ZZ}{\mathbbm{Z}} 
 
\newcommand{\HH}{\mathbbm{H}}

\newcommand{\id}{\mathrm{id}}
\newcommand{\ad}{\mathrm{ad}}
\newcommand{\Ad}{\mathrm{Ad}}
\newcommand{\Mat}{\mathrm{Mat}}

\newcommand{\End}{\mathrm{End}}

\newcommand{\Iso}{\group{Iso}}
\newcommand{\Aff}{\group{Aff}}
\newcommand{\GL}{\group{GL}}
\newcommand{\PGL}{\group{PGL}}
\newcommand{\SL}{\group{SL}}
\renewcommand{\SS}{\group{S}}
\newcommand{\TT}{\group{T}}

\newcommand{\SO}{\group{SO}}
\newcommand{\SU}{\group{SU}}
\newcommand{\UU}{\group{U}}
\newcommand{\OO}{\group{O}}

\newcommand{\Aut}{\group{Aut}}
\newcommand{\Hol}{\group{Hol}}

\newcommand{\Zen}{\mathrm{Z}}
\newcommand{\Nor}{\mathrm{N}}
\newcommand{\Inn}{\mathrm{Inn}}

\newcommand{\group}{\mathrm} 

\newcommand{\ag}[1]{\boldsymbol{#1}} 

\newcommand{\aL}{\ag{L}}
\newcommand{\aH}{\ag{H}}
\newcommand{\aU}{\ag{U}}
\newcommand{\aT}{\ag{T}}
\newcommand{\aG}{\ag{G}}

\newcommand{\ac}[1]{\overline{#1}^{\rm z}}

\renewcommand{\rho}{\varrho}
\renewcommand{\tilde}{\widetilde}
\renewcommand{\bar}{\overline}

\newcommand{\longto}{\ \longrightarrow\ }

\renewcommand{\epsilon}{\varepsilon}

\newcommand{\zsp}{\mathbf{0}} 
\newcommand{\met}{\langle\cdot,\cdot\rangle}
\newcommand{\one}{\{e\}}

\DeclareMathOperator{\im}{\mathrm{im}}

\DeclareMathOperator{\Span}{\mathrm{span}}

\DeclareMathOperator{\supp}{\mathrm{supp}}

\newcommand{\Tan}{\mathrm{T}}

\newcommand{\frg}{\mathcal{G}}
\newcommand{\frh}{\mathcal{H}}
\newcommand{\fra}{\mathcal{A}}
\newcommand{\frb}{\mathcal{B}}

\newcommand{\frn}{\mathcal{N}}
\newcommand{\frz}{\mathcal{Z}}

\newcommand{\frk}{\mathcal{K}}
\newcommand{\frj}{\mathcal{J}}

\newcommand{\frf}{\mathcal{F}}

\newcommand{\fro}{\mathcal{O}}
\newcommand{\fru}{\mathcal{U}}
\newcommand{\frs}{\mathcal{S}}

\newcommand{\frp}{\mathcal{P}}

\newcommand{\frt}{\mathcal{T}}
\newcommand{\frr}{\mathcal{R}}

\newcommand{\zen}{\mathcal{Z}}
\newcommand{\nor}{\mathcal{N}}

\newcommand{\gl}{\mathcal{GL}}

\newcommand{\der}{\mathcal{Der}}

\newcommand{\g}{\mathsf{g}}
\newcommand{\h}{\mathsf{h}}

\newcommand{\om}{\omega}

\DeclareMathOperator{\J}{\mathsf{J}}

\renewcommand{\d}{\mathrm{d}}
\renewcommand{\i}{\mathrm{i}}
\newcommand{\e}{\mathrm{e}}
\renewcommand{\phi}{\varphi}
\newcommand{\perpo}{{\perp_\omega}}
\renewcommand{\mod}{\bmod}

\newcommand{\gs}{\frg_{\rm s}}

\newcommand{\cf}{cf.\ }

\begin{document}

\baselineskip=0.48cm


\title[Pseudo-Hermitian homogeneous spaces of finite volume]{Rigidity of pseudo-Hermitian homogeneous spaces of finite volume}

\author[Baues]{Oliver Baues}
\address{Oliver Baues, Department of Mathematics, Chemin du Mus\'ee 23, University of Fribourg, CH-1700 Fribourg,  Switzerland} 
\email{oliver.baues@unifr.ch}

\author[Globke]{Wolfgang Globke}
\address{Wolfgang Globke, Faculty of Mathematics, University of Vienna, Oskar-Morgenstern-Platz 1, 1090 Vienna, Austria}
\email{wolfgang.globke@univie.ac.at}

\author[Zeghib]{Abdelghani Zeghib}
\address{Abdelghani Zeghib, \'Ecole Normale Sup\'erieure de Lyon, 
Unit\'e de Math\'ematiques Pures et Appliqu\'ees,
46 All\'ee d'Italie,
69364 Lyon,
France}
\email{abdelghani.zeghib@ens-lyon.fr}

\date{\today}


\begin{abstract}
Let $M$ be a pseudo-Hermitian homogeneous space of finite volume.
We show that $M$ is compact and the identity component $G$ of the group of
holomorphic isometries of $M$ is compact. If $M$ is simply connected, then
even the full group of holomorphic isometries is compact.
These results stem from a careful analysis of the Tits fibration of
$M$, which is shown to have a torus as its fiber.
The proof builds on foundational results on the automorphisms
groups of compact almost pseudo-Hermitian homogeneous spaces.
It is known that a compact homogeneous pseudo-K\"ahler manifold
splits as a product of a complex torus and a rational homogeneous variety,
according to the Levi decomposition of $G$.
Examples show that compact homogeneous pseudo-Hermitian manifolds in
general do not split in this way.
\end{abstract}

\maketitle

\tableofcontents


\section{Introduction and main results}
\label{sec:intro}

\subsection{Some directions in Gromov's vague pseudo-Riemannian conjecture}

Understanding when the isometry group of a pseudo-Riemannian 
structure acts non-properly, particularly when the isometry group
$\Iso(M,\g)$ of a compact pseudo-Riemannian manifold $(M,\g)$ is
non-compact, is a paradigmatic case of Gromov's vague conjecture stipulating 
that rigid geometric structures with large automorphism groups are
classifiable.
It is also a fundamental question in geometric dynamics that has been
investigated in many works with much progress, especially in the lower
index case, for example the Lorentzian index \cite{AS,DAG,zeghib} or
the authors' results on index two in \cite{BGZ}.
Indeed, let us observe for instance that for any semi\-simple Lie group $G$
with a uniform lattice $\Gamma$, the Killing form endows $M= G/\,\Gamma$
with  a pseudo-Riemannian metric whose isometry group contains $G$, acting
by left-translations, so that $M$ is a homogeneous pseudo-Riemannian
manifold.
However, these are not the only homogeneous pseudo-Riemannian manifolds,
which in general have the form $G/H$ where $H$ is not discrete.
Herein lies the difficulty of the problem.
So, even with an additional homogeneity hypothesis, the pseudo-Riemannian
problem stays intractable.


\subsubsection{Complex framework}

Our approach here in testing Gromov's vague
conjecture is to enrich the pseudo-Riemannian structure by assuming the
existence of a compatible complex structure.
In other words, we consider \emph{pseudo-Hermitian structures}.
More precisely, a \emph{pseudo-Hermitian metric} $\g$ on a complex manifold
$(M, \J)$ is a pseudo-Riemannian metric compatible with $\J$:
$\g(\J\cdot, \J\cdot) = \g(\cdot,\cdot)$.
Here we are interested in the automorphism group of such a structure, which
is the group of holomorphic transformations preserving the pseudo-Hermitian
metric.

 
 
%
%
%
To illustrate the effect of the complex structure, let us observe that the
previously introduced pseudo-Riemannian homogeneous examples $G/\,\Gamma$ cannot be
pseudo-Hermitian.
Indeed, the induced complex structure and pseudo-Hermitian product on the
Lie algebra $\frg$ of $G$ are $\Ad(\Gamma)$-invariant and hence
$\Ad(G)$-invariant by the Borel Density Theorem (here we assume $G$ to be
semisimple without compact factors).
For any $x\in\frg$, the adjoint operator $\ad(x)$ is skew with
respect to the pseudo-Hermitian product, and $\ad(\J x) = \J \ad(x)$. Now, the point is that the Lie algebra $\fru(p, q)$ of the unitary group of a
pseudo-Hermitian product of signature $(p, q)$ on $\CC^n$
is totally real in $\Mat(n,\CC)$, that is,
$\fru(p, q) \cap \i \fru(p, q) = \zsp$.
 
The main result of the present article (see Subsection \ref{results} below)
is the classification of
compact pseudo-Hermitian homogeneous spaces $G/H$ by showing that $G^\circ$
must be compact (compare the discussion following Corollary \ref{mcor:hermitian}).
In the initial example above, $H$ was discrete and we used the
Borel Density Theorem.
All this fails dramatically in the general case, where we
essentially have to deal with ``savage'' groups that are semidirect
products of a compact group by a solvable one.

\subsubsection{Pseudo-K\"ahler case}

The fundamental two-form associated to $\g$ is the differential two-form
defined by $\omega (\cdot,\cdot) = \g(\cdot, \J\cdot)$.
The metric is said to be \emph{pseudo-K\"ahler} if $\omega$ is closed. 
Conversely, from the symplectic geometry point of view, pseudo-K\"ahler structures
are special symplectic structures given by a symplectic form $\omega$
calibrated by
a complex structure $\J$,  that is 
$\omega(\J\cdot, \J\cdot) = \omega (\cdot,\cdot)$. 

Dorfmeister and Guan \cite{DG} proved the compactness of the identity
component of the
holomorphic isometry group of a homogeneous compact pseudo-K\"ahler
manifold, and that the Levi decomposition of this group induces a
holomorphic and metric splitting of the manifold.
The rigidity here comes essentially from the symplectic side.
In fact, Dorfmeister and Guan's theorem is to a large part implied by
Zwart and Boothby's \cite{ZB} results on compact symplectic homogeneous
spaces (compare Appendix \ref{sec:newproof}).


In the pseudo-K\"ahler case,
Gromov's vague conjecture seems to be tractable.
More generally, symplectic actions of finite-dimensional Lie algebras
on compact symplectic manifolds are ``easy'' to handle.
The basic idea
is that the
derived subalgebra consists of Hamiltonian vector fields, from which one
concludes that it is a direct sum of a compact semisimple and an abelian
Lie algebra (see for instance Guan \cite[p.~3362]{guan}, but also
\cite[Section 26]{GS}, \cite[Section 2]{huckleberry}).


\subsection{Statement of the main results}\label{results}

We begin our investigation by considering the more general situation of
\emph{almost} pseudo-Hermitian homogeneous manifolds. Such a manifold
carries a field of complex structures that is not required to be integrable
to a complex structure on the manifold.
Our first main result is:

\begin{mthm}\label{mthm:Levi_compact}
Any almost pseudo-Hermitian homogeneous space $M$ of finite
volume is compact. The maximal semisimple subgroup of the identity component
of the almost complex isometry group of $M$ is compact,
and its solvable radical is nilpotent and at most two-step nilpotent. 
\end{mthm}

It is therefore sufficient to henceforth consider compact spaces.

\medskip

In Section \ref{sec:spaces} we find that the only almost
pseudo-Hermitian compact homogeneous manifolds with a transitive action
by a solvable group are complex tori.

\begin{mthm}\label{mthm:solvable}
Let $M$ be a compact almost pseudo-Hermitian homogeneous space with
a transitive effective action by a connected solvable Lie group $G$ of almost complex isometries.
Then $G$ is abelian and compact.
Furthermore, $M$ is a flat complex pseudo-Hermitian torus
$\CC^{p,q}/\,\Gamma$, where $\Gamma$ is a lattice in $\CC^{p,q}$.
\end{mthm}

Upon making the assumption that $M$ is a complex manifold,
we can obtain much stronger statements.
In Section \ref{sec:complex}, we make use of
Tits's \cite{tits} result that
any compact complex homogeneous manifold fibers holomorphically over a
generalized flag manifold with complex parallelizable fibers.
We eventually derive the following structure theorem for holomorphic
isometry groups:

\begin{mthm}\label{mthm:torus_fiber} 
Let $M$ be a compact complex homogeneous manifold, and let $G$ be a
closed Lie subgroup of the holomorphic automorphism group of $M$.
Assume that $G$ preserves a
pseudo-Hermitian metric and acts transitively on $M$.
Then the fiber of the Tits fibration of $M$ is a compact complex torus
and the group $G$ is compact.
\end{mthm}

In particular, we obtain that
the identity component of the holomorphic isometry group of a
compact homogeneous pseudo-Hermitian manifold is compact.
It is not clear if this result is still valid if we merely assume an
\emph{almost} pseudo-Hermitian structure on the compact
homogeneous space.
\medskip

Even though we show in Theorem \ref{mthm:torus_fiber} that the identity
component of the holomorphic isometry group is compact, the full
holomorphic isometry group is not necessarily so.
The following example illustrates this.

\begin{example*}
Let $T=\CC^{p,q}/\ZZ^{2p+2q}$ be a complex torus with a pseudo-Hermitian
pro\-duct $\h$ of signature $(p, q)$. The identity component of the
holomorphic isometry group is $T$ itself.
However, the full holomorphic isometry group contains
$\Lambda(\h)= \UU(\h)  \cap\SL(2p+2q,\ZZ)$, where $\UU(\h)$ denotes the unitary group of $\h$. 
In the case  of the standard pseudo-Hermitian product
$\h(z,w)= -z_1\bar{w}_1-\ldots-z_p\bar{w}_p + z_{p+1}\bar{w}_{p+1}+\ldots+z_{p+q}\bar{w}_{p+q}$,
this group $\Lambda$ is a lattice in $\SU(\h)$  and thus non-compact (for
$p,q>0$). This fact extends to all pseudo-Hermitian products with rational
coefficients (by the Borel-Harish-Chandra Theorem \cite{BHC}). 
\end{example*}

However, in the simply connected case, Theorem \ref{mthm:torus_fiber} extends
to the full holo\-morphic isometry group.

\begin{mcor}\label{mcor:hermitian}
The holomorphic isometry group of a simply connected
compact homogeneous pseudo-Hermitian manifold is compact.
\end{mcor}

A class of examples illustrating Corollary \ref{mcor:hermitian} are
the compact homogeneous pseudo-Hermitian manifolds with non-zero Euler
characteristic.
These are precisely the generalized flag manifolds (cf.~Example
\ref{ex:euler_non0}).

\medskip

\label{geometric.splitting}
Theorem \ref{mthm:torus_fiber} can be interpreted as the identity component
of the holomorphic isometry group of a compact  homogeneous pseudo-Hermitian
manifold being inessential  in the sense that it preserves a positive
definite Hermitian structure. Thus, to classify these objects, start with a
compact complex  manifold homogeneous under the action of  a compact group
$K$, and take any pseudo-Hermitian product invariant under the isotropy of
some point.
By compactness of this isotropy, there always exists an invariant Hermitian product, but existence of higher signature ones depends on the isotropy representation. 
The compact group splits, up to finite index, as $G = T \times K$, where 
$T$ is abelian and $K$ is semisimple.
In the pseudo-K\"ahler case, this splitting induces a holomorphic and
metric splitting of the manifold as a product of a complex torus and a flag
manifold (a projective homogeneous manifold), see
Borel and Remmert \cite{BR}, Dorfmeister and Guan \cite{DG}.
Without the assumption that the fundamental two-form is closed, we do not
have such a splitting.
This is illustrated by the examples discussed at the end of Section
\ref{sec:complex}.
 
\medskip

In Appendix \ref{sec:newproof}, we review the results by Zwart and Boothby \cite{ZB}
on compact symplectic homogeneous spaces and by Dorfmeister and Guan \cite{DG} on
compact homogeneous pseudo-K\"ahler manifolds.
We give short proofs based on the techniques developed in this article.
Moreover, in Theorem \ref{thm:metric_symplectic}, we prove a structure
theorem for the intermediate case of a compact symplectic homogeneous space
equipped with an invariant pseudo-Riemannian metric, but without the
assumption of compatibility between the metric and the symplectic form.

\subsection{Additional remarks beyond the homogeneous case}

%
In order to illustrate how our current homogeneity hypothesis is fairly
reasonable, let us show that any semisimple real Lie group $G$ can act
(non-transitively, but topologically transitive) by preserving a
pseudo-Hermitian structure on a locally homogeneous compact manifold.
Let $\kappa$ be the Killing form on the Lie algebra $\frg$.
On its complexi\-fication $\frg^\CC= \CC \otimes \frg$, consider the 
quadratic form $q(x + \i y) = \kappa(x, x) + \kappa(y, y)$, where
$x+ \i y \in \frg^\CC$. 
Polarization yields an $\ad(\frg)$-invariant pseudo-Hermitian form.
This allows us to define a left-invariant pseudo-Hermitian metric $\h$ on the group
$G^{\CC}$,
that equals $q$ at the identity element.
By the $\ad(\frg)$-invariance of $q$, $\h$ is invariant under the 
right-action of $G$ on $G^\CC$.
If $\Gamma$ is a uniform lattice of $G^\CC$, then $\h$ projects to a
well-defined $G$-invariant pseudo-Hermitian metric on $M= G^\CC /\,\Gamma$.
By Moore's Ergodicity Theorem (see Zimmer \cite[Theorem 2.2.6]{zimmer}),
the $G$-action is ergodic and hence topologically transitive.

These examples are not pseudo-K\"ahler, because, as mentioned above, a Lie
group preserving a symplectic structure on a compact manifold has a product of a
compact group by an abelian group as its derived subgroup.

If $G = \SL(2,\RR)$, then $G^\CC= \SL(2,\CC)$, which
is locally isomorphic to
$\SO(1, 3)$, and so $M$ is the frame bundle of the
hyperbolic three-manifold (or orbi\-fold)
$\HH^3 /\,\Gamma = \SO(3) \backslash \SO(1,3) /\,\Gamma$. 
The pseudo-Hermitian metric in this case is of Hermite-Lorentz type, and
in particular has index two when viewed as a pseudo-Riemannian metric.

The whole discussion generalizes to the case where $\kappa$ is replaced 
by any $\ad(\frg)$-invariant scalar product and $G$ is any group admitting such a form,
sometimes called a quadratic group (there is a huge family of such groups,
see, for instance Kath and Olbrich \cite{KO1}).
If $G$ is an oscillator group, then, as in the $\SL(2,\RR)$-case,
one gets a Hermite-Lorentz example.
These beautiful and rich examples motivate
the syste\-matic study of isometric actions of Lie groups on
pseudo-Riemannian homogeneous manifolds of index two as initiated by the
authors in \cite{BGZ}.

\subsection*{Notations and conventions}

All Lie algebras $\frg$ under consideration are assumed to be finite-dimensional
and defined over the reals (although most of our algebraic results are valid over any 
 field of characteristic zero). 

For direct products of Lie algebras $\frg_1$, $\frg_2$ we write
$\frg_1\times\frg_2$, whereas $\frg_1+\frg_2$ or $\frg_1\oplus\frg_2$
refer to sums as vector spaces. 

The \emph{solvable radical} $\frr$ of a Lie algebra $\frg$ is the maximal solvable
ideal of $\frg$.  Any semisimple Lie algebra $\frf$ over the real numbers
is a direct product $\frf=\frk\times\frs$, where $\frk$ is a semisimple Lie algebra of \emph{compact type}, meaning its Killing form is definite, and where  $\frs$ is semisimple without factors of compact type.

Any maximal semisimple Lie subalgebra $\frf = \frk \times \frs$ of $\frg$ is called a \emph{Levi subalgebra} of $\frg$. Let us further put  
 $\gs=\frs\ltimes\frr$. Then there is  a Levi decomposition 
\[
\frg =  \frf \ltimes \frr = (\frk\times\frs)\ltimes\frr = \frk \ltimes \gs . 
\]

For a Lie group $G$, we let $G^\circ$ denote the connected component
of the identity.
For a subgroup $H$ of $G$, we write $\Ad_{\frg}(H)$ for the adjoint
representation of $H$ on the Lie algebra $\frg$ of $G$, to distinguish
it from the adjoint representation $\Ad(H)$ on its own Lie algebra
$\frh$. The closed subgroup $H$ is called \emph{uniform} if $G/H$ is compact.

The center of a group $G$, or a Lie algebra $\frg$, is denoted by
$\Zen(G)$, or $\zen(\frg)$, respectively.
Similarly, the centralizer of a subgroup $H$ in $G$
(or a subalgebra $\frh$ in $\frg$) is denoted by $\Zen_G(H)$
(or $\zen_\frg(\frh)$).
The normalizer of $H$ in $G$ is denoted by $\Nor_G(H)$, that of $\frh$ in $\frg$
by $\nor_{\frg}(\frh)$.

\subsection*{Acknowledgements}

Wolfgang Globke was partially supported by the Australian Research Council grant {DE150101647} and the Austrian Science Fund FWF grant I 3248.


%
%
%

\section{$h$-structures on Lie algebras}
\label{sec:prelim}

Here we work out the elementary linear algebra regarding the definition of almost $h$-algebras and 
$h$-structures on Lie algebras (compare  Definitions \ref{def:h_algebra},  \ref{def:h_structure}, and  Section \ref{sec:h_Lie}.). Such algebraic structures serve as local models for isometry groups of certain geometric manifolds.   

\subsection{$h$-structures} 

\subsubsection{Bilinear forms}
Let $\beta$ be any bilinear form  on  $V$. We shall denote orthogonality of vectors in $V$ with respect to $\beta$ by the symbol $\perp_\beta$. If $\beta$ is a symmetric or skew symmetric  bilinear form,  the subspace 
$$ V^{\perp_\beta} = \{ v \in V \mid \beta(v,V) = \zsp \}$$
is called the \emph{kernel} of $\beta$. Accordingly, $\beta$ is called
\emph{non-degenerate} if
$V^{\perp_\beta} = \zsp$. 

Furthermore,  we  put
\[
\fro(\beta) \,   = \,  \{ \varphi \in \gl(V) \mid  \beta(\varphi(x),y) =  - \beta(x, \varphi(y)) \}
\]
for the Lie algebra of linear transformations that are skew with respect to $\beta$.  The form $\beta$ 
will be called \emph{invariant} by $\varphi \in \gl(V)$ if $\varphi \in \fro(\beta)$. 

\subsubsection{$J$-structures}

A linear map  $J \in \End(V)$ will be a called a \emph{$J$-structure} on $V$
if $J^{2} \equiv   -\id_{V} \mod \ker J$. 
We let 
$$   \gl_{J}(V) =   \{ \varphi \in \gl(V) \mid  \varphi \, (\ker J)  \subseteq  \ker J   \text{ and } 
J \circ \varphi \equiv \varphi \circ J \mod \ker J  \} $$
denote the Lie algebra of $J$\/-linear maps of $V$.  For any $J$-structure,  
\begin{equation} \label{eq:kerJ}   V = \ker J \oplus \im J . 
\end{equation}  
A $J$-structure  is  a complex structure  if $\ker J = \zsp$.

\subsubsection{$h$-structures on vector spaces}

Suppose that $V$ is equipped with a $J$-structure. Let  $\met$ be a  $J$-invariant symmetric bilinear form with kernel $V^{\perp} $. 

\begin{definition} \label{def:h_structure}
The data $(J, \met)$ will be called an \emph{$h$-structure} on $V$ if 
$$   \ker  J  \, \subseteq \, V^\perp  . $$ 
\end{definition}

In this case, we also have
\[
\langle J u, J v \rangle = \langle u, v \rangle \  \text{ for all $u,v \in V$.}
\]
The associated $J$-invariant skew-symmetric form
$$  \om( \cdot , \cdot) =  \langle J \cdot , \cdot \rangle$$  
is  called the \emph{fundamental two-form of the $h$-structure}. Note that  
$$   V^{\perp} =   V^{\perpo} .  $$
We call  $V^{\perp}$ the \emph{kernel of the $h$-structure}.
Given an $h$-structure on $V$, the symbol $\perp$ shall \emph{always} 
denote orthogonality with respect to $\met$.  

\begin{remark}
	If the $h$-structure is non-degenerate (that is, if $\met$ has trivial kernel),  then  $(V, J, \met)$ is  a   \emph{Hermitian  vector space} with complex structure $J$. We do not require that $\met$ is definite in the definition of a Hermitian vector space.
\end{remark} 

\subsubsection{Unitary Lie algebra of $h$-structures}

For any $h$-structure on $V$ we put  
\begin{equation} \label{su_0}
  \fro(V): =  \fro(\met) \   \text{ and }  \  \frs\frp(V): =  \fro(\om)  . 
 \end{equation} 
Furthermore 
\begin{equation} \label{su_1}
  \fru(V): =   \fro(V) \cap  \gl_{J}(V)  \ \ 
  ( =    \frs\frp(V)  \cap  \gl_{J}(V) ) 
 \end{equation} 
 is called the \emph{unitary Lie algebra}  for the $h$-structure.  Note that
 \begin{equation} \label{su_2}
      \fru(V) \  \subseteq \   \fro(V) \cap  \frs \frp(V),
       \end{equation} 
and that  equality holds in \eqref{su_2} if  $\ker J = V^{\perp}$.  

\subsubsection{Skew linear maps in Hermitian vector spaces} 

Let  $(V, J, \met)$ be a Hermitian vector space. 
 
\begin{lem}\label{lem:JandO}
Let  $\varphi \in \fro(V) \cap   J  \fro(V)$. Then we have: 
\begin{enumerate}
\item
$ J \varphi +  \varphi J = 0 $.
\item
If $\varphi$ and $J \varphi$ commute, then $\varphi^{2} = 0$. 
\end{enumerate}
\end{lem}
\begin{proof}
Let $\varphi^{*} \in \End(V)$ denote the dual map of $\varphi$ with respect to $\met$.
Since $J, \varphi  \in \fro(V)$, $(J \varphi)^{*} = \varphi J$. Since also 
$ J  \varphi \in   \fro(V)$, $(J \varphi)^{*} =  -J \varphi $. Hence, (1) holds.
This implies $ J \varphi^{2} + \varphi J \varphi = 0$. If $\varphi$ and $J \varphi$ commute 
we obtain $J \varphi^{2} = 0$. Hence, $ \varphi^{2} = 0$. This shows (2). 
\end{proof} 

A real subspace $W$ of a vector space with complex structure $(V,J)$ is called \emph{totally real} if $W \cap J W = \zsp$.  The unitary Lie algebra of a Hermitian vector space is totally real. More generally: 

\begin{lem}  \label{lem:su_1}
 $  \fru(V) \cap J  \fro(V)  = \zsp$. 
 \end{lem}
\begin{proof}
Indeed, let $\varphi \in \fru(V)$  with $\varphi = J \psi$, where $\psi \in \fro(V)$. 
By Lemma \ref{lem:JandO} (1), $\psi$ is complex anti-linear. However, $\psi = - J \varphi \in  \gl_{J}(V)$.
Thus $\psi = 0$. 
\end{proof}  

\subsection{$h$-structures on Lie algebras} \label{sec:h_Lie}

We now consider $h$-structures on Lie algebras. 
We start by introducing some notions related to bilinear forms and $J$-structures on Lie algebras. 

\subsubsection{Lie algebras with bilinear forms} 

Let $\frg$ be a Lie algebra and  $\beta$ any bilinear form on $\frg$. The subalgebra of $\beta$-skew elements of $\frg$ is
\[
\frg_\beta\, =\, \{x\in\frg\mid \beta([x,y],z)=-\beta(y,[x,z])\text{ for all } y,z\in\frg\}.
\]
For a subalgebra $\frh$ of $\frg$, we say that $\beta$ is \emph{$\frh$-invariant}
if  $\frh \subseteq \frg_{\beta}$.

\begin{lem}\label{lem:idealp}  
Let $\frj \subseteq \frg_{\beta}$ be an ideal of $\frg$. 
Then $[\frj^{\perp_{\beta}}, \frj ] \subseteq  \frj \cap \frg^{\perp_{\beta}}$. 
\end{lem}

In case  $\beta$ is  symmetric or skew-symmetric, the pair $(\frg, \beta)$ is called \emph{effective} if the kernel  $\frg^{\perp_{\beta}}$  does not contain any non-trivial ideal of $\frg$. 

\subsubsection{$J$-structures on Lie algebras} 

Given a $J$-structure on the Lie algebra $\frg$ we put $ \frh = \ker J$, which is a subspace of $\frg$.
Define the subalgebra 
\[
\frg_{J}\, =\,  \{x\in\frg\mid  \ad(x) \frh  \subseteq \frh \text{ and }  \ad(x)   J y  =   J\,\ad(x)   y  \mod \frh \} .
\]
We observe: 
\begin{lem}\label{lem:Jlinear}
If $x\in\frg$, $y \in\frg_{J}$ then
$J  (\ad(x)\, y) =  \ad(J x) y  \mod \frh $. 
\end{lem}
\begin{proof}
$
J  (\ad(x)   y)  =   - J   (\ad(y) \, x)  =  -  \ad(y)  J x \mod \frh 
= \ad(J x) y \mod \frh$. 
\qedhere
\end{proof}

%
%

%
\subsubsection{$h$-structures on Lie algebras} 

Let $(\met,J)$ be  an  $h$-structure on $\frg$, and let  $\om$ denote its 
fundamental two-form. We  call
\[
\frg^{\perp}:  =   \frg^{\perp_{\met}}  =  \frg^{\perp_{\omega}}
\]
the \emph{kernel of the $h$-structure} $(\met,J)$. For an $h$-structure on $\frg$ we define 
\[
\frg_{J}^{h}\, =\,  \{x\in\frg\mid  \ad(x) \frg^{\perp}  \subseteq \frg^{\perp} \text{ and }  \ad(x)   J y  =   J  \ad(x)   y  \mod \frg^{\perp }\} .
\]
We then have: 
\begin{lem}\label{lem:Jlinear2}
If $x\in\frg$, $y \in \frg_{J}^{h}$ then
$
J  (\ad(x)\, y) =  \ad(J x)  y  \mod \frg^{\perp} $. 
\end{lem}

We further  write  
\[
\frg_{h}=\frg_\omega\cap\frg_{\met} \ \text{ and } \  \frg_{\met,J}= \frg_{\met} \cap \frg_{J} 
\]
Observe that 
\begin{equation} \label{eq:Jlinear2}
\frg_{\met,J} \,  \subseteq  \,  \frg_{h} \, \subseteq  \, \frg_{J}^{h} \; . 
\end{equation}
From \eqref{su_2} we further infer
\begin{equation} \label{eq:gsu_2}
\frg_{h}  \, \subseteq \, \frg_{J} = \frg_{J}^{h}  \;  \text{   and }  \; 
\frg_{\met,J} =  \frg_{h} \  \text{ if } \ker J = \frg^{\perp} .
\end{equation} 






\newcommand{\spl}{\mathrm{spl}}
\newcommand{\frT}{\mathfrak{T}}
\newcommand{\Tens}{\mathbf{T}}
\renewcommand{\theta}{\vartheta}

\section{Algebraic dynamics arising from automorphism groups of geometric structures}
\label{sec:invariance}

In this section, we explore the dynamics arising from automorphism groups of multilinear
forms on manifolds of finite volume. 
\subsection{Preliminaries:  isotropic and split subgroups of algebraic groups}

\subsubsection{Solvable groups} 
For a solvable linear algebraic group $\ag{A}$ defined over $\RR$, 
let $\ag{A}_\spl$ denote the maximal $\RR$-split connected subgroup of
$\ag{A}$. This means that $\ag{A}_\spl$ is the maximal connected subgroup
trigonalizable over the reals.
For the real algebraic group $A = \ag{A}_{\RR}$, its maximal trigonalizable
subgroup  is $A_{\spl} =  A \cap \ag{A}_\spl$. 
Let $\aT$ be a torus  defined over $\RR$. Then $\aT$ is called
\emph{anisotropic}  if $\ag{T}_\spl=\one$.  Otherwise, $\aT$ is called \emph{isotropic}. 
Equivalently,  $\aT$ is anisotropic if its group of real points
$T=\aT_{\RR}$ is compact.  Every torus defined over $\RR$ has a
decomposition into subgroups
$ \aT=\aT_\spl\cdot \aT_{\rm c}$, 
where $ \aT_{\rm c}$ is a maximal anisotropic torus defined over
$\RR$ and $\aT_\spl \cap \aT_{\rm c}$ is finite.
Moreover, if $\aT \subseteq \ag{A}$ is a maximal torus defined over $\RR$
and $\aU$ is the unipotent radical of $\ag{A}$,
then there is a semidirect product decomposition
\[
\ag{A}_{\spl} = \aU \cdot \aT_{\spl}.
\]
Note also that the split part $\ag{A}_{\spl}$ is preserved under
$\RR$-defined morphisms of algebraic groups. 
See Borel \cite[\textsection 15]{borel_book} for more background.

\subsubsection{Semisimple groups}
A simple linear algebraic group $\aL$ defined over $\RR$ is called 
\emph{isotropic} if any maximal $\RR$-defined torus is isotropic.
A semisimple linear algebraic group $\aL$ defined over $\RR$ will be called 
\emph{isotropic} if all of its simple factors are isotropic. 

\subsubsection{The isotropic kernel} 
If $\aG$ is a linear algebraic group defined over $\RR$
with solvable radical $\ag{R}$ and a maximal semisimple subgroup $\ag{L}$,
then we define the \emph{isotropic part}  of $\aG$ to be the normal subgroup
\begin{equation}
\aG_{\rm is} = \ag{S}\cdot\ag{R}_{\spl} \, ,
\label{eq:splitpart}
\end{equation}
where  $\ag{S}$ is the maximal normal isotropic subgroup 
of $\aL$.  Note that $\aG_{\rm is}$ is a normal subgroup and that there is an almost
semidirect product decomposition 
\[
\aG  = \ag{B} \cdot   \aG_{\rm is} \, ,
\]
where  $\ag{B}$ is a reductive anisotropic subgroup. 
The normal subgroup $ \aG_{\rm is}$ will be called the \emph{isotropic kernel}
of $\aG$. Any $\RR$-defined homomorphism of $\aG$ to an anisotropic 
group contains $\aG_{\rm is}$ in its kernel, that is, it factors over
$\aG/\aG_{\rm is}$. 

\begin{lem}\label{lem:split0}
Let $\aG$ and $\aH$ be $\RR$-defined linear algebraic groups and
$\psi:\aG\to\aH$ an $\RR$-defined algebraic homomorphism.
Then $\psi(\aG_{\rm is})\subseteq\aH_{\rm is}$. 
\end{lem}
\begin{proof}
The split part of the solvable radical of $\aG$ is mapped to the split part
of the solvable radical of $\aH$ (see Borel \cite[15.4 Theorem]{borel_book}).
Moreover, $\psi$ maps the isotropic simple subgroups of $\aG$
to isotropic simple subgroups of $\aH$.
\end{proof}

\subsubsection{Real algebraic groups}

Let $G = \aG_{\RR}$ be the group of real points of $\aG$.
Then $G_{\rm is} =  G \cap  \aG_{\rm is}$ 
is called the isotropic kernel  of $G$ and $B = \ag{B} \cap G$ is a compact group. We note
that  $$    G =  B \cdot G_{\rm is}  \;. $$  
We conclude that  $G_{\rm is}$ is a normal uniform subgroup of $G$, and, in fact, every morphism of 
real algebraic groups from  $G$ to a compact group factors over 
$G/G_{\rm is}$. 
Recall that a closed subgroup $H$ of $G$ is called \emph{uniform}  if $G/H$ is compact.

\subsection{Dynamics of linear maps, invariant probability measures}

Let  $V$ be a real vector space and $G \subseteq  \GL(V)$.  
Let  $A= \ac{G}_{\RR}$ denote the real Zariski closure of $G$.  

\subsubsection{Quasi-invariant vectors} \label{sec:quasi-invariant}
For $v \in V$, put 
$G_{v} = \{ g \in G \mid g \cdot v = v \}$. Recall that  $v$ is 
called \emph{invariant} (by $G$) if $G = G_{v}$.

\begin{definition} \label{def:quasi-invariant}
We call $v \in V$ \emph{quasi-invariant by $G$} if the closure of the orbit $G \cdot v$ 
is compact. Moreover, $v$ is called  \emph{strongly quasi-invariant} if $\ac{G}_{\RR} \cdot v$ has compact closure in $V$. 
\end{definition} 

Clearly, for any group $G$, strong quasi-invariance implies quasi-invariance. 
If $G \subseteq \GL(V)$ is a real linear algebraic group, then quasi-invariance and 
strong quasi-invariance of $v \in V$ are equivalent. 

\begin{example} \label{ex:trig}
Let $T  \subseteq \GL(V)$ be a trigonalizable subgroup.  
Then $v \in V$ is quasi-invariant  by $T$ if and only if $v$ is 
invariant by $T$.  
\end{example} 

\begin{lem} \label{lem:Gis}
Let $G \subseteq \GL(V)$ be a real linear algebraic group and $v \in V$. 
Put $H  = G_{v}$. Suppose that $v$ is quasi-invariant by $G$. Then $H$ contains 
the isotropic kernel  $G_{\rm is}$ of $G$. In particular, $H$ is a uniform subgroup of $G$. 
\end{lem}
\begin{proof}  By Example \ref{ex:trig},  $H = G_{v}$ must contain every trigonalizable subgroup of $G$. 
Therefore,  $H$ contains $G_{\rm is}$. 
 \end{proof}
 
 \subsubsection{Invariant probability measures} 
We now consider  $G$-invariant probability measures on $V$.

\begin{example}[Invariant measures are supported on fixed points] \label{ex:trio}
Let $A \subseteq \GL(V)$ be a one-dimensional trigonalizable group and $\mu$ an
$A$-invariant probability measure.   If $\mu$ is ergodic then, by Zimmer \cite[Proposition 2.1.10]{zimmer},  
$\mu$ is supported on some orbit $A \cdot v$. Since the Haar-measure on $A$ is 
infinite, we must have $A \cdot v = v$. By the ergodic decomposition theorem 
\cite[Chapter 5]{viana-oliveira}, any invariant 
probability measure $\mu$ may be approximated by finite convex combinations of ergodic invariant 
measures, which (as we just remarked) are all supported on the fixed point set $V^{A}$. Therefore,  any $A$-invariant probability measure $\mu$ is supported on  $V^A$.   
\end{example}

 The following is a key observation in the context of invariant probability measures 
 on representation spaces: 
 
 \begin{prop} \label{prop:measure}
 Let $G \subseteq  \GL(V)$ and $\mu$ a $G$-invariant probability measure on $V$. 
 Put  $A = \ac{G}_{\! \RR}$. Then the following hold: 
\begin{enumerate}
\item $\mu$ is invariant by $A $.
\item
The action of $A$ on the support of $\mu$ factors over a compact group. 
%
\item
There exists a $G$-invariant subspace 
$W \subseteq V$ containing the support of $\mu$,
such that the action of $G$ on $W$ factors over a compact group.  
\item
If $A =  A_{\rm is}$ then $\mu$ is supported on $V^A$.  
\end{enumerate} 
\end{prop} 
\begin{proof} 
Embed  $V$ as an affine subspace into some real 
projective space $\mathrm{P}^n\RR$.
For any probability measure $\nu$ on $\mathrm{P}^n\RR$, 
let  $\PGL(n)_{\nu}$ denote the stabilizer of $\nu$  in the
projective linear group. By \cite[Theorem 3.2.4]{zimmer}, $\PGL(n)_{\nu}$
is a real algebraic subgroup of $\PGL(n)$. Since $\mu$ pushes forward to a $G$-invariant measure 
on $\mathrm{P}^n\RR$, this implies (1).

To prove (4), observe that  $ A_{\rm is}$ is generated by its one-dimensional split subgroups. 
In light of Example \ref{ex:trio}, we deduce that $\supp \mu$ is contained
in the  fixed point set of $A_{\rm is}$.
In particular, $A_{\rm is}$ acts trivially on the support of $\mu$, so (2) follows.   
Let $W$ be the smallest $G$-invariant subspace of $V$ 
that contains $\supp \mu$.
Then $A_{\rm is}$ acts trivially on $W$. Hence (3) holds.
\end{proof}
 
Note that the action of $G$ on $W$ factors over a compact group and the elements of $W$ are 
therefore strongly quasi-invariant by $G$.  

\subsection{The characteristic map of a geometric $G$-manifold}

Let $M$ be a differentiable manifold, and $\theta$ an $s$-multilinear
form on  $M$. 
The group of diffeo\-morphisms of $M$ that
preserve $\theta$ will be  denoted by $\Aut(M,\theta)$. Let 
$$  G\subseteq\Aut(M,\theta)$$ be a Lie group.
We may identify the Lie algebra $\frg$ of $G$ with the subalgebra of 
vector fields on $M$ whose flows are one-parameter subgroups of
diffeomorphisms contained in $G$.
We let $L^s(\frg)$ denote the space of $s$-linear forms on $\frg$,
and consider the associated  map 
\[
\Phi:M \to L^s(\frg),\quad
p \mapsto \Phi_{p}
\]
where, for $X_1,\ldots,X_s\in\frg$, 
the multilinear form  $\Phi_{p}$ on  $\frg$ is declared by
\[
\Phi_{p}(X_1,\ldots,X_s)
= \theta_{p}(X_{1,p},\ldots, X_{s,p}) \; . 
\]
The adjoint representation $\Ad$ of $G$ on $\frg$ induces a representation
$$ \rho:G\to\GL(L^s(\frg))   \; .  $$
The map $\Phi$ is equivariant with respect to the 
action of $G$ on $M$ and the representation $\rho$. 

\begin{thm}\label{thm:isom_is_quasi-invariant}
Let $M$ be a differentiable manifold  and  $G\subseteq\Aut(M,\theta)$.
Suppose that the Lie group $G$ preserves a finite measure on $M$. Then the action 
of $G$ on $$ \Phi(M) \, \subseteq  \, L^s(\frg)$$  factors over a compact group. 
Moreover, the forms $\Phi_{p}$, $p \in M$, are strongly quasi-invariant by $G$. 
\end{thm}
\begin{proof}
Put $V=L^s(\frg)$ and
$A={\ac{\rho(G)}}_{\! \RR}$.
Since $\Phi$ is $G$-equivariant, the finite $G$-invariant measure $\mu$ on $M$ pushes forward  
to a finite $G$-invariant measure $\mu'$ on $V$, which is supported on $\Phi(M)$. By Proposition 
\ref{prop:measure} (3), there exists an $A$-invariant linear subspace $W$ containing 
$\Phi(M)$ on which the action of $A$ factors over a compact group. In particular, 
$\Phi_{p}$ is strongly quasi-invariant for all $p \in M$. 
\end{proof}

\subsection{Nil-invariant and quasi-invariant bilinear forms on Lie algebras}
\label{sec:nilinvariant}

Let $\frg$ be a Lie algebra and consider its  inner auto\-morphism group  $$G = \Inn(\frg) \, . $$  
By definition, 
$G$ is the connected Lie subgroup of $\GL(\frg)$ that is tangent to the Lie algebra  of linear 
transformations 
$$ \ad(\frg) = \{ \ad(x)  \mid x \in \frg \} \; \subseteq \; \gl(\frg) \, . $$ 
We also consider the real Zariski closure $$\ac{\Inn(\frg)}  \, \subseteq \, \Aut(\frg) \, \subseteq \,  \GL(\frg)$$ of $\Inn(\frg)$, which is an algebraic group of automorphisms of $\frg$ with  Lie algebra $$\ac{\ad(\frg)} \,  \subseteq \der(\frg) \, \subseteq \,  \gl(\frg) \, .$$

We are interested in the  action of $G$ on the space of bilinear forms on $\frg$. The following notions are fundamental (\cf Section \ref{sec:quasi-invariant}): 

\begin{definition} \label{def:quasi-invariant_beta}
Let $\beta$ be a  bilinear form on $\frg$. Then 
$\beta$ is called \emph{invariant} if $G$ is contained in the orthogonal group  $\OO(\beta)$ of
$\beta$. Likewise,  $\beta$ will be called \emph{quasi-invariant}  if 
$$ \ac{\Inn(\frg)} \cap \OO(\beta)   \; \,  \subseteq \; \,  \ac{\Inn(\frg)}$$ 
is the inclusion of a uniform subgroup. 
\end{definition} 

According to Lemma \ref{lem:Gis}, $\beta$ is  quasi-invariant  if and only if 
\begin{equation} \label{eq:quasi}
 \ac{\Inn(\frg)} \cap \OO(\beta) \, \supseteq  \; \,   \ac{\Inn(\frg)}_{\rm is} \, . 
\end{equation}
That is, 
the stabilizer of $\beta$ contains the maximal normal subgroup 
$$       \ac{\Inn(\frg)}_{\rm is} $$ that is  generated by linear maps in  $\ac{\Inn(\frg)}$
that are trigonalizable over $\RR$. 

\subsection{The quasi-invariance condition} \label{sec:quasi_invariance}
The condition \eqref{eq:quasi}  may be formulated in Lie algebra terms  only: 

\begin{lem} \label{lem:Gisbeta}
The bilinear form $\beta$ is quasi-invariant by $G$ if and only if the Lie algebra $$ \ac{\ad(\frg)} \cap  \fro(\beta)$$  contains all elements of\/  $\ac{\ad(\frg)}$ that are trigonalizable over $\RR$.  
\end{lem}

The Lie algebra  $\ac{\ad(\frg)}_{\rm is}$ generated by all elements trigonalizable over $\RR$ is tangent to 
$  \ac{\Inn(\frg)}_{\rm is}$. In particular,
$\ac{\ad(\frg)}_{\rm is}$ is an ideal in  $\ac{\ad(\frg)}$, since
$\ac{\Inn(\frg)}_{\rm is}$ is a normal subgroup of  $\Inn(\frg)_{\rm is}$.
The following  is a noteworthy  consequence: 

\begin{lem} \label{lem:acis} 
Suppose that $\beta$ is a quasi-invariant bilinear form. Then $$  \ac{\Inn(\frg)}_{\rm is} (\frg)  \subseteq \frg_{\beta}. $$ 
 That is, for any $\varphi \in  \ac{\ad(\frg)}_{\rm is}$ and $a \in \frg$, we have $\varphi(a) \in 
\frg_{\beta}$.
 \end{lem}
 \begin{proof}
 Since  $\ac{\ad(\frg)}_{\rm is}$ is an ideal of $\ac{\ad(\frg)}$, we have  
 $$ [\ac{\ad(\frg)}_{\rm is}, \ad(\frg)] \,  \subseteq \,   \ac{\ad(\frg)}_{\rm is} . $$
 But $\ad(\frg)$ is also an ideal 
 in the Lie algebra $\der(\frg)$ of all derivations of $\frg$. This implies that, for any 
 $\varphi \in \ac{\ad(\frg)}_{\rm is}$, 
$$ [\varphi, \ad(a)] = \ad (\varphi(a)) \,  \in  \, \ad(\frg) \cap \ac{\ad(\frg)}_{\rm is} \; . $$ 
If $\beta$ is  quasi-invariant, 
we have  $\ac{\ad(\frg)}_{\rm is} \subseteq \fro(\beta)$. Therefore, for any $\varphi \in \ac{\ad(\frg)}_{\rm is}$, 
$\ad(\varphi(a)) \in \ac{\ad(\frg)}_{\rm is} \subseteq  \fro(\beta)$. In particular, $\varphi(a) \in \frg_{\beta}$. 
\end{proof}

\subsection{The nil-invariance condition} \label{sec:nil_invariance}

The following condition is somewhat weaker than quasi-invariance since it takes into account only the nilpotent elements in the endomorphism Lie algebra $\ac{\ad(\frg)}$:

\begin{definition} \label{def:nilinvariant_beta}
A bilinear form $\beta$ on $\frg$ is called \emph{nil-invariant} if 
$$     \varphi \in   \fro(\beta)   $$ 
for all nilpotent elements $\varphi$ of the Lie algebra  $\ac{\ad(\frg)}$.
\end{definition} 

We fix  a Levi decomposition 
\[
\frg = (\frk\times\frs)\ltimes\frr , 
\]  
where $\frk$ is semisimple of compact type, $\frs$ is semisimple without
factors of compact type, and $\frr$ is the solvable radical of $\frg$.

\begin{lem} If $\beta$ is nil-invariant, then 
\begin{equation} \label{eq:nilinvariant}
	\frs + \frn \, \subseteq \, \frg_\beta \, .
\end{equation}
\end{lem}
To see that $\frs$ is contained in $\frg_\beta$ recall that the subalgebra $\frs$ is generated by the set of elements $x \in \frs$ such that $\ad_\frg(x)$ is nilpotent (for example \cite[Lemma 5.1]{BGZ}).

\subsection{Quasi-invariant and nil-invariant symmetric bilinear forms}

In the following, we recall important structural properties of 
Lie algebras with nil-invariant symmetric bilinear form.  

\begin{thm}[\mbox{\cite[Theorem A, Corollary C]{BGZ}}]\label{thm:BGZ_invariance}
Let $\frg$ be a finite-dimensional real Lie algebra with
nil-invariant symmetric bilinear form $\met$.
Let $\met_{\gs}$ denote the restriction of $\met$ to $\gs = \frs \ltimes \frr$.
Then:
\begin{enumerate}
\item
$\met_{\gs}$ is invariant by the adjoint action of $\frg$ on $\gs$.
\item
$\met$ is invariant by the adjoint action of $\gs$.
\end{enumerate}
Furthermore, if $(\frg, \met)$ is effective,
then
\begin{enumerate}
\item[(3)] $ \frg^{\perp} \, \subseteq \, \frk \ltimes \frz(\gs)$  and $\; [\frg^{\perp}, \gs] \, \subseteq \, \frz(\gs) \cap \frg^{\perp}$. If  $\frg^\perp \subseteq \frg_{\met}$ then $[\frg^\perp , \gs] = \zsp$.
\end{enumerate}
\end{thm}

By definition, quasi-invariance always implies nil-invariance. It is somewhat surprising and non-trivial that for a \emph{symmetric} bilinear form also the converse holds and an even stronger implication is satisfied in the solvable case:  

\begin{cor}  \label{cor:nil_and_quasi}
Let $\beta$ be a symmetric bilinear form on $\frg$. Then $\beta$ is quasi-invariant if and only if $\beta$ is nil-invariant.  Moreover, if $\frg$ is solvable then any nil-invariant symmetric bilinear form $\beta$ is
invariant. 
\end{cor} 

Corollary  \ref{cor:nil_and_quasi} is directly implied by Theorem \ref{thm:BGZ_invariance}, which streng\-thens 
\eqref{eq:nilinvariant} for symmetric bilinear forms considerably.

\subsection{Nil-invariant and quasi-invariant skew-symmetric bilinear forms}

As noted above for symmetric bilinear forms on Lie algebras, nil-invariance already implies quasi-invariance.
The following simple example illustrates
that, in general,  this is not the case for skew-symmetric bilinear forms: 

\begin{example} \label{ex:nil_noquasi}
Let $\frg$ be the two-dimensional solvable Lie algebra with generators
$x,y$, satisfying $[x,y]=y$.
Consider the skew-symmetric form $$\omega = \d x \wedge \d y$$  on $\frg$
(which satisfies
$\omega(x,y)=1$).
It is easily verified that $\omega$ is invariant by $\ad(y)$, but not by
$\ad(x)$.
This means $\omega$ is nil-invariant without being quasi-invariant.
\end{example}

In the following Section \ref{sec:skew_Lie}, we will develop the structure theory of quasi-invariant and nil-invariant skew-symmetric forms in more detail.


\section{Nil-invariant and quasi-invariant $h$-structures}
\label{sec:hermitian_nil}
This section is devoted to the study of quasi-invariant $h$-structures. 

\subsection{Lie algebras with  skew-symmetric form} \label{sec:skew_Lie}

Let $\om$ be a skew-symmetric bilinear form on the Lie algebra $\frg$. 

\begin{lem}\label{lem:Jnilinvariant2}
Let  $x, z\in\frg$, $y \in\frg_\omega$.  Then 
\begin{equation}
\omega([x,y],z)  = \omega( [z,y], x)  \;. 
\label{eq:prolo}
\end{equation}
If also $z \in \frg_\omega$ then 
\begin{equation}
\omega(\ad(x)y,z)
=
\omega(y,\ad(x)z),
\label{eq:adxyz}
\end{equation}
In particular,
\begin{equation}
[\frg_{\omega},\frg_{\omega}]\perp_{\omega} \frg_{\omega}.
\label{eq:adxyz2}
\end{equation}
\end{lem}
\begin{proof}
Since  $y \in\frg_{\omega}$,
$
\omega([x,y],z) = -\omega(x,[z,y]) = \omega( [z,y], x) 
$.

\noindent 
Hence, if  $z \in \frg_{\omega}$ then 
\[ \omega([x,y],z) 
 = \omega([z,x],y)
= \omega(y,[x,z]).
\]
Thus \eqref{eq:adxyz} holds. If furthermore  $x\in\frg_{\omega}$, then
\[
\omega([x,y],z)=-\omega(y,[x,z]) .
\]
So \eqref{eq:adxyz} implies \eqref{eq:adxyz2}.
\end{proof}

\begin{lem}\label{lem:phigg}
Let $\varphi \in  \fro(\om)  $  satisfy
$\varphi(\frg) \subseteq  \frg_{\om}$. Then 
$\varphi([\frg_\omega,\frg_\omega])  \,  \subseteq \, \frg^{\perpo}$. 
\end{lem}
\begin{proof}
Let $x,y \in \frg_{\om}$ and $a\in\frg$. Then 
$\omega(\varphi([x,y]),a)
= - \omega([x,y], \varphi(a) )$. 
Since $ \varphi(a)\in \frg_{\om}$,  \eqref{eq:adxyz2} implies
$\varphi([\frg_\omega,\frg_\omega])  \, \perpo \,  \frg$.  
\end{proof} 

\subsubsection{Invariant skew-symmetric forms}

Recall that the skew-symmetric form $\om$ is called invariant if $\frg= \frg_\om$, and that $(\frg, \om)$ is called \emph{effective} if $\frg^\perpo$ does not contain any non-trivial ideal of $\frg$. We deduce from \eqref{eq:adxyz2}: 

\begin{prop}\label{prop:skew_invariant_abelian}
Let $\om$ be an  invariant skew-symmetric form. Then we have
$$ [\frg, \frg]  \, \subseteq  \,  \frg^{\perp_\om} . $$  
In particular, if $(\frg, \om)$ is effective,  then $\frg$ is abelian and $\om$ is non-degenerate. 
\end{prop}

\subsubsection{Nil-invariant skew-symmetric forms} \label{sec:nil_skew}

We now assume that $\om$ is nil-invariant (in the sense of Definition \ref{def:nilinvariant_beta}). 
This implies in particular that the  nilradical $\frn$ of $\frg$,   and also any semisimple subalgebra  $\frs$ of non-compact type are  contained in  $\frg_{\om}$.  In fact, as we show now,  any semisimple subalgebra  $\frs$ of non-compact type  is contained  in the kernel of ${\omega}$:

\begin{prop}[Levi subalgebra is of compact type] 
\label{prop:Sorthogonal}
Let $\omega$ be a nil-invariant skew-symmetric form on $\frg$. Let $\frj(\frs)$ be an ideal of $\frg$ generated by a semisimple subalgebra $\frs$ of non-compact type. Then  $$  \frj(\frs) \subseteq \frg^\perpo \, . $$
In particular, if $(\frg, \om)$ is effective,  then  every semisimple subalgebra of $\frg$ is 
of compact type.
\end{prop}
\begin{proof} 
Recall that any derivation of the solvable radical $\frr$ of $\frg$ has image in $\frn$
(\cite[Theorem III.7]{jacobson}). Since the semisimple Lie algebra $\frs$ acts reductively, 
we infer that  $[\frs,\frr]=[\frs,\frn]$. 
We may assume that the image of $\frs$ is an ideal in $\frg/\frr$. Therefore    
$$ [\frs, \frg] = \frs + [\frs, \frn]\, \subseteq \, \, [\frg_{\om}, \frg_{\om}] . $$
By Lemma \ref{lem:phigg},  $  \frs + [\frs, \frn] \,  \subseteq \, \frg^{\perpo}$.  
Then we have 
$$[\frg, [\frs, \frn]]  =   [\frs, \frn] +  [\frn, [\frs, \frn]] . $$ 
By induction, the 
ideal $\frb$ of $\frg$ generated by $[\frs,\frn]$ equals 
the ideal  of $\frn$ generated by  $[\frs,\frn]$. 
Since $[\frn,\frg^\perpo] \subseteq  \frg^\perpo $, we deduce  that $\frb\subseteq\frg^\perpo$.
Hence,  $\frj(\frs) = \frs + \frb \subseteq \frg^\perpo$.
\end{proof}

Let $\frr$ denote the solvable radical of $\frg$. We write $\frr_\omega=\frr\cap\frg_\omega$.
Since $\frn\subseteq\frr_{\omega}$, $\frr_{\omega}$ is an ideal of $\frg$. 

\begin{lem}\label{lem:2nilpotenta}  
\hspace{1ex} 
\begin{enumerate}
\item $
[\frr_\omega,[\frg_\omega,\frg_\omega]] \,  \subseteq \, \frg^{\perpo}$. 
\item
$[\frn,\frn] \perp_\om \frn$. 
\item $[[\frn,\frn]^{\perp_{\om}}, [\frn, \frn] ] \subseteq \frg^{\perpo}$. 
\item $[\frg,\frn] \perp_\om  \frz_{\frg}(\frn)$.  
\end{enumerate}
\end{lem}
\begin{proof}
Since $[\frr_\omega, \frg ]  \subseteq \frn \subseteq\frg_\omega$,
Lemma \ref{lem:phigg} implies (1).  Whereas (2) follows 
from  \eqref{eq:adxyz2}, and (3) from Lemma \ref{lem:idealp}. 
Since $\frn \subseteq \frg_{\om}$, (4) is immediate.  
\end{proof}


It then follows: 

\begin{lem}\label{lem:2nilpotentb}
Let $(\frg, \omega)$ be effective. Then 
\begin{enumerate}
\item $\frr_{\omega} = \frn$.
\item $\frn$ is at most two-step nilpotent (that is, $[\frn, \frn] \subseteq \frz(\frn)$). 
\end{enumerate}
\end{lem}
\begin{proof} Since $[\frg, \frr_\om] \subseteq \frn \subseteq \frr_\om$, $\frr_\om$ is an ideal in $\frg$. By Lemma \ref{lem:2nilpotenta}, the ideal $[\frr_\omega,[\frr_\omega,\frr_\omega]]$ is contained in $\frg^{\perpo}$. Therefore, $[\frr_\omega,[\frr_\omega,\frr_\omega]] = \zsp$.
Hence, $\frr_\omega$ is a nilpotent ideal of $\frg$, so that   $\frr_\omega \subseteq  \frn$. Hence (1) and (2) are satisfied. 
\end{proof} 

\begin{prop} \label{pro:nilpotent_isab}
Let $\omega$ be a nil-invariant skew symmetric form,  such that  $(\frg, \omega)$ is effective.
If $\frg$ is nilpotent then $\frg$ is abelian.  
\end{prop}
\begin{proof}  
Since $\frg = \frn$ is nilpotent, Lemma \ref{lem:2nilpotenta} (2) implies that $[\frn,\frn] \subseteq \frn^{\perpo}$ is an ideal of $\frn$. Since $(\frn, \om)$ is effective this implies $[\frn, \frn] = \zsp$. 
\end{proof}

We summarize the above as follows: 

\begin{thm} \label{thm:nil_skew}
Let $\om$ be a nil-invariant skew-symmetric form, such that $(\frg, \om)$ is  effective. Then the following hold: 
\begin{enumerate}
\item Any Levi subalgebra of $\frg$ is of compact type. 
\item The nilradical $\frn$ of $\frg$ is at most two-step nilpotent. 
\item If $\frg$ is nilpotent, then $\frg$ is abelian. 
\end{enumerate}
\end{thm} 

We remark that the nilradical of a Lie algebra with nil-invariant
skew-symmetric form $\om$ may be non-abelian, even for  non-degenerate $\om$: 
 
\begin{example}[Solvable algebra with nilradical $\frh_{n}$ of Heisenberg type] \label{ex:nilinv_nonabelian}
Consider a $2n$-dimensional 
symplectic vector space  $(W, \om_{0})$. We extend $\omega_{0}$ to a  $2n + 2$-dimensional 
symplectic vector space $$   (\frg_{n} = \fra \oplus W \oplus \frb , \, \om)   $$
where $\fra = \Span a$ and $\frb =  \fra^{*} = \Span z$ are dually paired,  one-dimensional vector spaces. 
For the definition of $\om$  we require: 
$$ W \, \perpo  \, (\fra \oplus  \frb) , \quad
\omega(a, z) =  \langle a, z \rangle  = 1
\quad \text{ and }  \quad \omega = \omega_{0}  \text{ on }  W. $$  
Next, turn $\frg$ into a Lie algebra with the non-zero Lie brackets satisfying 
\begin{align*}  
[x, y ]   =  \omega_{0}(x,y)  z ,  \quad
[a,x]  =    x, \quad
[a,z]  =  2z , \quad \text{ where }    x,y  \in W .
\end{align*} 
Thus $\frg$ is a split  extension of the one-dimensional Lie algebra by a $2n+1$ dimensional  Heisenberg Lie algebra 
$$\frh_{n} = W \oplus \frb . $$  
The element $a \in  \frg$ acts on $\frh_{n}$ by the  derivation $D:  \frh_{n}  \to \frh_{n}$, defined by $Dx = x$, $x \in W$, $D z = 2z$. 
We observe further that $\omega$ is a nil-invariant \emph{non-degenerate} skew-symmetric form on $\frg$. For the nil-invariance property of $(\frg, \om)$  it is sufficient to show that, for all $v \in \frh_{n}$, 
 the operators $\ad(v) : \frg \to \frg$ are skew with respect to $\om$.  
 This follows mainly by the equations
\begin{align*}  
 \omega([a, x], y )   =  \omega(x,y) =  \omega(a,  [x, y] ) ,\text{ for all }  x,y \in W .
\end{align*} 
\end{example}


%
\subsubsection{Quasi-invariant skew-symmetric forms}

By Definition \ref{def:quasi-invariant_beta} and Lemma \ref{lem:Gisbeta} a skew-symmetric form $\omega$ on the Lie algebra $\frg$ is quasi-invariant
if and only if all trigonalizable elements in the smallest algebraic Lie
algebra containing $\ad(\frg)$ are skew with respect to $\om$.
Clearly,  quasi-invariance of 
$\om$ implies nil-invariance but the converse is not true, as already observed in Example \ref{ex:nil_noquasi}.

\begin{example}
Any quasi-invariant  skew-form $\omega$ on a solvable Lie algebra $\frg$ of real type (where the adjoint representation has only real eigenvalues) must  be invariant. Thus Proposition \ref{prop:skew_invariant_abelian}  shows that for any effective pair $(\frg, \omega)$, where $\om$ is quasi-invariant and $\frg$ solvable of real type,  $\frg$ is abelian. In particular, the family of nil-invariant almost symplectic solvable Lie algebras $(\frg_n,\om )$ from Example \ref{ex:nilinv_nonabelian} is not quasi-invariant. 
 \end{example}
 
The following property distinguishes quasi-invariant and  nil-invariant skew-sym\-metric forms:

%
%
\begin{prop}  \label{pro:prolongation}
Suppose that $\om$ is quasi-invariant. Then 
$$ [\frg, [\frn, \frn]]  \, \subseteq \,  [\frn, \frn] \cap \frg^{\perpo} . $$ 
\end{prop} 

\begin{proof} 
Define $U =  \frg / [\frn,\frn]^{\perpo}$ and $V= [\frn, \frn] /([\frn, \frn] \cap \frg^{\perpo})$. 
By \eqref{eq:prolo} and (3) of Lemma \ref{lem:2nilpotenta}, the  map 
$$ (u,u',v) \mapsto f(u,u',v) :=  \om([u, v], u'),  \quad u,u' \in \frg, \, v \in [\frn,\frn]$$ 
descends to an element $\bar f \in \SS^{2} U^{*} \otimes V^{*}$. Observe that 
$\om$ induces a dual pairing of $U$ and $V$, and thus also an
isomorphism $U \to V^{*}$. Hence, we may view $\bar f$
as an element of $\SS^{2} U^{*} \otimes U$. 

For $u\in\frg$,
decompose $\ad(u) = \varphi_{\rm i} + \varphi_{\rm s}$ as
a sum of commuting derivations of $\frg$, where $ \varphi_{\rm i}$ is semisimple with
 purely imaginary eigenvalues and $\varphi_{\rm s}$ is trigonalizable over the reals (see the remark following below).
The assumption that $\om$ is quasi-invariant implies that $\varphi_{\rm s}  \in  \fro(\om)$.
Moreover,  by Lemma  \ref{lem:acis},   $\varphi_{\rm s}(\frg) \subseteq \frg_{\om}$.
 Hence, Lemma \ref{lem:phigg} shows that 
 $$\varphi_{s} ([\frn, \frn]) \, \subseteq\, [\frn, \frn] \cap \frg^{\perpo}. $$   
Therefore,  $f(u,u',v) =  \om( \varphi_{\rm i} (v), u')$. 

We deduce that  $\bar f$ defines an element of the first prolongation of the orthogonal Lie algebra $\fro(\met)$ for some suitable positive definite scalar product $\met$ on $U$. Since the orthogonal Lie algebras  have trivial first prolongation (\cf \cite[Chapter I, Example 2.5]{kobayashi}), we deduce that
$\bar f = 0$.
\end{proof}

\begin{remark}[Jordan decomposition in algebraic Lie algebras] For any element $\varphi \in \gl(\frg)$, let $\varphi = \varphi_{\rm s} + \varphi_{\rm n}$ denote its Jordan decomposition, where $\varphi_{\rm s} \in \gl(\frg)$ is semisimple, $\varphi_{\rm n} \in \gl(\frg)$ is nilpotent, and $$ [\varphi_{\rm s}, \varphi_{\rm n}] = 0 . $$ Suppose now that $\varphi \in \ac{\ad(\frg)}$. Since the Lie algebra $\ac{\ad(\frg)}$ is the Lie algebra of an algebraic group, it also contains the Jordan parts $\varphi_{\rm s} , \varphi_{\rm n}$ of $\varphi$. Moreover, as follows for example from \cite[\S 8.15]{borel_book},  $\varphi_{\rm s} =  \varphi_{\rm i} + \varphi_{\rm split}$ decomposes uniquely as a sum of commuting semisimple derivations both contained  in  $ \ac{\ad(\frg)}$ such that  the eigenvalues of $ \varphi_{\rm i}$ are imaginary and $ \varphi_{\rm split}$ is diagonalizable over $\RR$. 
 \end{remark}

\begin{prop}   \label{pro:prolongation_effective}
Suppose that $\om$ is quasi-invariant. If $(\frg, \om)$ is effective then the following hold: 
\begin{enumerate}
\item $ [\frn, \frn] \cap \frg^{\perpo} = \zsp$.
\item $[\frn, \frn] \subseteq \, \frz(\frg)$.
\end{enumerate} 
\end{prop} 
\begin{proof} By Proposition \ref{pro:prolongation}, $[\frg,  [\frn, \frn] ] \subseteq [\frn, \frn] \cap  \frg^{\perpo} $. In particular, $ [\frn, \frn] \cap  \frg^{\perpo} $ is an ideal in $\frg$. 
By effectiveness of $(\frg,\om)$,  (1) and also (2) follow. 
\end{proof} 

\begin{lem}  \label{lem:lsym}
Let $\om$ be a  nil-invariant  skew-symmetric form on $\frg$. Then, for all $a \in \frr$,  $n,n' \in \frn$,  we have:
$$    [  [a, n], n' ]  - [ n, [ a, n']]  \,   \in  \, \frg^{\perpo} . $$  
\end{lem}
\begin{proof} Let $b \in \frg$. Note that $[a, b] \in \frn$, since $a \in \frr$. Therefore,  by Lemma \ref{lem:2nilpotenta} (2),
 \begin{equation} \label{eq:commuten} 
   \om(  [  [a, b] , n] ,  n'   ) =  0  \; . \end{equation} 
 Now \begin{equation*}  \begin{split}  \om(  [  a ,  [ b , n] ],  n'   ) &=    \om(a,  [ [ b , n] ,  n' ]  ) = \om( [n, [n', a] ], b) 
\;  \text{ and }  \\ 
  \om(   [  b,  [ a , n] ],  n'   )  &=    \om(b,  [ [ a , n] ,  n' ] ) \;  . \end{split} 
  \end{equation*} 
   In the view of   \eqref{eq:commuten}, we deduce
$ \om (b,  [n, [a, n'] ])  =   \om(b,  [ [ a , n] ,  n' ] )$. 
\end{proof}

\begin{prop}  \label{prop:lsym_and_der}
Suppose that $\om$ is quasi-invariant. If $(\frg, \om)$ is effective then  $$ [\frr, \frn] \, \subseteq \, \frz(\frn). $$   
Moreover, this implies $  [\frn, \frn] \subseteq \frr^\perpo$
and $[\frr, \frn] \subseteq \frn^{\perpo}$. 
\end{prop}
\begin{proof}
Let $a \in \frr$, $n, n'  \in \frn$. By Lemma \ref{lem:lsym},
$  [  [a, n], n' ]  - [ n, [ a, n']]    \in  [\frn, \frn] \cap \frg^{\perpo}$.
According to 
Proposition \ref{pro:prolongation_effective}, $ [\frn, \frn] \subseteq \frz(\frg)$.
In particular, 
$ [\frn, \frn] \cap \frg^{\perpo}$ is an ideal in $\frg$. By effectiveness, $ [\frn, \frn] \cap \frg^{\perpo} = \zsp$.
Hence, $    [  [a, n], n' ]  =   [ n, [ a, n']]  $. 

On the other hand, $ [\frn, \frn] \subseteq \frz(\frg)$  and the Jacobi identity imply 
$[[ a, n ], n'] + [n, [a, n']] = 0$. We deduce  $ [  [a, n], n' ]  =  -  [ n, [ a, n']]  $. 
Hence,  $ [  [a, n], n' ]  = 0$. 

Let $r \in \frr$ and let $\ad(r) = \ad(r)_{\rm s} + \ad(r)_{\rm n}$ be its Jordan decomposition. To show that $r \, \perpo \, [\frn, \frn]$, decompose $\frn = W + \frz(\frn)$, with $\ad(r)_{\rm s}W = \zsp$. In particular,  $[\frn, \frn] = [W,W]$.  Let $u,v  \in W$. Then $\om(r, [u,v]) = \om([r,u],v) = \om( \ad(r)_{n} u, v) = - \om(u,  \ad(r)_{\rm n} v) =  - 
\om(u, [r, v]) =  - \om(r, [u,v])$.  Therefore, $\om(r, [u,v]) = 0$. 
\end{proof}

\begin{thm} \label{thm:skew_solvable}
Let $\om$ be a quasi-invariant skew-symmetric form on a Lie algebra
 $$ \frg = \frk \times \frr . $$  Assume further that $(\frg, \om)$ is effective. Then the following hold:
\begin{enumerate} 
\item
The nilradical $\frn$ of $\frg$ is abelian. 
\item
The solvable radical 
$\frr$ of $\frg$  is of imaginary type (meaning that the eigenvalues of the adjoint representation are imaginary).
\end{enumerate} 
 \end{thm}
\begin{proof} 
By Proposition \ref{prop:lsym_and_der},  $[\frn, \frn] \subseteq \frr^{\perpo}$. Since $\frk$ and $\frn$ commute, $[\frn, \frn] \perp \frk$. This shows  $[\frn, \frn]   \subseteq  \frg^{\perpo}$. Since $(\frg, \om)$ is
effective,  $[\frn, \frn] = \zsp$. 
Hence (1) holds.

Let  $a \in \frr$, and as above decompose the Jordan semisimple part of $\ad(a)$ as a canonical sum of commuting semisimple derivations:
$\ad(a)_{\rm s} =  \ad(a)_{\rm i} + \ad(a)_{\rm split}$.
 By quasi-invariance of $\om$,  $\varphi = \ad(a)_{\rm split}$ is a skew derivation of $(\frg, \om)$. 

By Proposition \ref{prop:lsym_and_der},  $[\frr, \frn] \subseteq \frz(\frn) \cap \frn^{\perpo}$ and therefore 
$\varphi(\frn) = 
\varphi (  \frz(\frn) \cap \frn^{\perpo} )$. Using that $\varphi$ is skew, we compute,  for any $n \in  \frz(\frn) \cap  \frn^{\perpo}$, $b \in \frr$,  that $\om(\varphi(n),b) = - \om(n, \varphi(b)) = 0$,
since $\varphi(b) \in \frn^{\perpo}$. 
Therefore, $\varphi(\frn) = \varphi( \frz(\frn) \cap \frn^{\perpo}) \subseteq \frr^{\perpo}$. From  $[\frk, \frr] = \zsp$, we infer $\varphi(\frk) = \zsp$, which implies $\frk \perp_{\om} \varphi(\frg)$ and  $\varphi(\frn) \subseteq \frg^{\perpo}$. 

Observe that $[b, \varphi (n) ] = \varphi ( [b,n])  + [ \ad(b), \varphi ] (n)$.  
Since $\frn$ is abelian, the restrictions of $\ad(a)$ and $\ad(b)$ to $\frn$ commute.
Hence, the Jordan part $\ad(a)_{\rm s}$ also commutes with $\ad(b)$.
Therefore, $\varphi$ commutes with $\ad(b)$. This implies $[ \ad(b), \varphi ] (n)  = 0$, showing  $[b, \varphi (\frn) ] \subseteq \varphi(\frn)$. Hence,  $\varphi(\frn)$ is an ideal in $\frg$. 
Since $(\frg, \om)$ is effective, we deduce that $\varphi(\frn) = \zsp$. Hence, $\varphi = 0$ and 
$\ad(a)_{\rm s} = \ad(a)_{\rm i}$, showing (2).
\end{proof}

\begin{cor} \label{cor:quasi_solv_om}
Let $\frr$ be a solvable Lie algebra with a quasi-invariant non-dege\-ner\-ate skew-symmetric 
form $\om$. Then $\frr$ is abelian.
\end{cor}
\begin{proof}
We may decompose $\frn = W \oplus (\frn \cap \frn^{\perpo})$ and
$\frg = \fra \oplus \frn$, such that $W^{\perpo} = \fra \oplus (\frn \cap \frn^{\perpo})$.  Since $\frn \cap  \frn^{\perpo} \perpo W$ and $\omega$ is non-degenerate, $\dim \fra \geq \dim \frn \cap  \frn^{\perpo}$. On the other hand,
$\fra$ is faithfully represented on  $\frn \cap  \frn^{\perpo}$ by an abelian Lie algebra  of derivations with imaginary eigenvalues only.  This shows $\dim \frn \cap  \frn^{\perpo} > \dim \fra$, unless $\dim \fra = 0$.  
We conclude $\fra =  \frn \cap  \frn^{\perpo} = \zsp$. Therefore, $\frg = W$ is abelian. 
\end{proof}


\subsection{Lie algebras with nil-invariant $h$-structure} 
\label{sec:nil_h-structure}

Let  $(\met,J)$ be an $h$-structure for the Lie algebra $\frg$. 
The  $h$-structure $$ (\frg, \met,J)$$
is called  \emph{nil-invariant} if \emph{both} the bilinear form $\met$ and the fundamental two-form $\om$ are nil-invariant. Similarly, the $h$-structure is called \emph{quasi-invariant} if both forms are quasi-invariant forms under the adjoint action of $\frg$.  

%
%

%
\subsubsection{Invariant $h$-structures}  

The $h$-structure  $(\met,J)$ on  $\frg$ will be called an \emph{invariant}  $h$-structure
if $\frg = \frg_{h}$. This  of course is the case if and only if $\met$ and $\om$ are invariant by $\frg$. 
In particular, for an invariant $h$-structure \eqref{eq:Jlinear2} implies that 
$\frg = \frg_{\J}^{h}$. Moreover,  $\frg^{\perp}$ 
is an ideal in  $\frg$. The following is a direct consequence of  
Proposition \ref{prop:skew_invariant_abelian}: 

\begin{prop}\label{prop:complex_invariant_abelian}
Let $(\met, \J)$ be an invariant $h$-structure for $\frg$.
Then  \begin{equation} \label{su_3L} [\frg, \frg] \;  \subseteq \;  \frg^{\perp} .  \end{equation} 
In particular, if $(\frg, \met)$ is effective, then $\frg$ is abelian. 
\end{prop}

\subsubsection{Nil-invariant $h$-structures}

Let  $(\frg, \met,J)$ be a nil-invariant $h$-structure for $\frg$. 
Let $\frn$ and $\frr$ denote the nilpotent and solvable radicals  of $\frg$. Nil-invariance implies that  $\frn \subseteq \frg_{\om}$, and also $\frr \subseteq  \frg_{\met}$ by Theorem \ref{thm:BGZ_invariance}. 
Together with  \eqref{eq:gsu_2} we thus have inclusions 
\[
\frn \, \subseteq \, \frr_\omega \, = \, 
\frr_{\met,\omega} \, \subseteq \, \frg_{h} \, \subseteq \, \frg_{J}^{h}\, .
\]

Since $\frn \subseteq \frg_{J}^{h}$,  Lemma \ref{lem:Jlinear2} implies: 
\begin{lem}\label{lem:JGN2}
$ J  [\frg, \frn]= [ J \frg,\frn] \mod \frg^{\perp}$ and $J \frn = \frn \mod \frg^{\perp}$.
\end{lem}

We then obtain:  

\begin{lem}  \label{lem:GJGN} 
 $[\frg_{J}^{h},  [\frg, \frn] ] \, \subseteq  \, [\frg, \frn]^{\perp}$.
\end{lem}
\begin{proof}
By Lemma \ref{lem:JGN2}, $J$ maps the subspace $ [\frg, \frn] + \frg^{\perp}$ of $\frg$ to itself. 
In fact, since $\ker J \subseteq \frg^{\perp}$ and $\frg^{\perp}$ is stable by $J$, $J$ induces a complex structure $J'$ on  
$ [\frg, \frn]  /  ([\frg, \frn] \cap \frg^{\perp})$ that is compatible with the symmetric bilinear form induced by $\met$. 
Thus  $V = [\frg, \frn]  / ([\frg, \frn] \cap   [\frg, \frn]^{\perp})$   is a Hermitian vector 
space.  As $\met$ is nil-invariant, Theorem \ref{thm:BGZ_invariance}  states  that, for any $x \in \frg$,  $\ad(x)$ induces a skew linear map $[\frg, \frn] \to [\frg, \frn]$. In particular, $\ad(x)$ descends to  a linear map  $\varphi_{x} \in  \fro(V)$. Moreover, since  $[\frg, \frn] \subseteq \frn \subseteq \frg_{J}^{h}$, Lemma \ref{lem:Jlinear2} implies $\varphi_{J x} = J' \varphi_{x}$. In particular, $J' \varphi_{x} \in \fro(V)$ is skew as well,
which implies that $\varphi_{x}$ is a complex anti-linear map (see Lemma \ref{lem:JandO} (1)).
If $x \in \frg_{J}^{h}$ then $\varphi_{x}$ is also complex linear, so that   $\varphi_{x} = 0$. 
 \end{proof}
 
\begin{remark}  One can prove in the same way the following stronger statement:  
 $[  \frg_{J}^{h}, \frn ] \, \subseteq  \, \frn^{\perp}$. 
\end{remark}

 
 \smallskip 
Let $a \in \frg_{\met}$. Choose a Fitting decomposition for the linear map
$\ad(a)$ of $\frg$. This  is the unique $\ad(a)$-invariant decomposition 
 $\frg = E_{0}(a) \oplus E_{1}(a)$, where $E_{0}(a) = \ker \ad(a)^{n}$, $n = \dim \frg$, 
 and $\ad(a): E_{1}(a) \to E_{1}(a)$ is a vector space isomorphism. Furthermore, $E_{0}(a) \perp E_{1}(a)$,
 since $\ad(a) \in \fro(\frg)$. This implies $E_{1}(a)^{\perp} \cap E_{1}(a) = \frg^{\perp} \cap E_{1}(a)$.

\begin{lem} \label{lem:E1} 
Let $a \in \frr$. Then $E_{1}(a)$ is contained in $[\frg, \frn]  \cap \frg^{\perp}$. 
 \end{lem}
\begin{proof} 
Let $\varphi_{a} \in \fro(V)$ denote  the map induced by $\ad(a)$, where  $V$ is as in the proof of Lemma \ref{lem:GJGN}. Since $\frn \subseteq \frg_{J}^{h}$, $\frn$ acts trivially on $V$ by Lemma \ref{lem:GJGN}. Since $[\frg, \frr] \subseteq \frn$, this implies that $\varphi_a$ commutes with $\varphi_{x}$ for all $x \in \frg$. In particular,  $\varphi_a$ commutes with $J \varphi_a = \varphi_{J a}$. Thus, as remarked in (2) of Lemma \ref{lem:JandO}, $\varphi_a^{2} = 0$. Since $\ad(a)$ is an isomorphism on $E_{1}(a)$, we infer that $E_{1}(a) \subseteq  [\frg, \frn]  \cap [\frg, \frn]^{\perp}$. In particular, $E_{1}(a) \subseteq   [\frg, \frn]^{\perp} \subseteq  E_{1}(a)^{\perp}$. By the above remark on the Fitting
decomposition of $\ad(a)$, this shows $E_{1}(a) \subseteq  [\frg, \frn]  \cap \frg^{\perp}$.
\end{proof}

 Our main result on nil-invariant almost $h$-structures is now:
 
\begin{thm}\label{thm:radical_nilpotent}
Let $\frg$ be a Lie algebra with a nil-invariant and effective almost
$h$-structure.
Then $\frg  = \frk \ltimes \frn$, where $\frk$ is compact semisimple. Moreover, the nilradical $\frn$ is at most two-step 
nilpotent.
\end{thm}
\begin{proof}  By Theorem \ref{thm:nil_skew}, we have $\frs = \zsp$. 
Now let $a \in \frr$. Lemma \ref{lem:E1} states that
$E_{1}(a)  \subseteq \frn \cap \frg^{\perp}$. By Theorem \ref{thm:BGZ_invariance} (3), $\frn \cap \frg^{\perp}$
is contained in $\frz(\frr)$. Since $a \in \frr$, this implies $E_{1}(a) = \zsp$. Hence, $\ad(a)$ is nilpotent. This shows that
$\frr = \frn$. By Lemma \ref{lem:2nilpotentb}, $\frn$ is at most two-step nilpotent.
\end{proof}

\begin{remark} Note that,  under the assumption of Theorem \ref{thm:radical_nilpotent}, also the various 
properties of Lemma \ref{lem:2nilpotenta} are satisfied for $(\frg, \om)$
\end{remark}

\begin{cor}\label{cor:solv_abelian}
Let $\frg$ be a solvable Lie algebra with a  nil-invariant and effective almost $h$-structure. Then $\frg$ is abelian and the almost $h$-structure is non-degenerate. 
\end{cor}
\begin{proof}  By Theorem \ref{thm:radical_nilpotent}, $\frg$ is in fact nilpotent. Therefore the $h$-structure is 
invariant. 
Thus $\frg$ is abelian by Proposition \ref{prop:complex_invariant_abelian}
\end{proof}

\section{Application to almost pseudo-Hermitian and almost symplectic
homogeneous spaces} \label{sec:spaces}

Given a differentiable manifold $M$ with almost complex structure $\J_M$,
a \emph{pseudo-Hermitian metric}  is a pseudo-Riemannian
metric $\g_M$ satisfying the compati\-bility condition
\[
\g_M(\J_M X,\J_M Y) = \g_M(X,Y)
\]
for all vector fields $X,Y$ on $M$. We call $(M, \J_{M}, \g_{M})$ an \emph{almost 
pseudo-Hermitian manifold.} 
If $\J_M$ is a complex structure, $M$ is called a pseudo-Hermitian manifold. 
Associated to $\g_M$ and $\J_M$ is the \emph{fundamental two-form} $\omega_M$  
which is given by
\[
\omega_M(X,Y) = \g_M(\J_M  X,Y).
\]



\subsection{Lie algebra models for pseudo-Hermitian homogeneous spaces}
\label{sec:hermitian_models}

Consider a homogeneous manifold $$ M = G/H$$ that admits a $G$-invariant  almost complex
structure $\J_M$. Let $\frg$ denote the Lie algebra of $G$ and $\frh$ 
the Lie subalgebra of $\frg$ corresponding to $H$.  
Then $\J_M$ can be described by (the choice of)  an endomorphism $J$ on the Lie algebra
$\frg$ with the following properties: 
\begin{enumerate}
\item[(J1)]
$J x=0$ for $x\in\frh$.
\item[(J2)]
$J^2 x = -x \bmod \frh$ for all $x\in\frg$.
\item[(J3)]
$J\Ad(h)x=\Ad(h) J x \bmod \frh$ for all $h\in H$, $x\in\frg$.
\end{enumerate}
Moreover,  $\J_M$ is a complex structure if and only if
(see e.g.~Koszul \cite[Section 2]{koszul})
\begin{enumerate}
\setcounter{enumi}{3}
\item[(J4)]
$[J x,J y]-[x,y]- J[x, J y]- J[ J x,y]\in\frh$ for all $x,y\in\frg$.
\end{enumerate}
If in addition $(M, \J_{M})$ has a $G$-invariant pseudo-Hermitian metric $\g_M$,
then $\g_M$ induces a symmetric bilinear form $\met$ on $\frg$ 
that  satisfies
\begin{enumerate}
\item[(H)]
$\langle x,y\rangle = \langle J x, J y\rangle$  for all $x,y\in\frg$. 
\end{enumerate}
Similarly, $\omega_M$ induces a skew-symmetric bilinear form $\omega$ on
$\frg$, satisfying
\[
\omega(x,y)=\langle J x, y\rangle.
\]

\begin{remark}
Note that since $\met$ is induced by $\g_M$, we have
\begin{equation}
\begin{split}
\frh
= \frg^\perp
&=
\{x\in\frg\mid \langle x,y\rangle = 0\text{ for all } y\in\frg\} \\
&=
\{x\in\frg\mid \omega(x,y) = 0\text{ for all } y\in\frg\}.
\end{split}
\label{eq:GperpH}
\end{equation}
\end{remark}

\subsubsection{The $h$-algebra model of a pseudo-Hermitian homogeneous space}

By con\-ditions (J1), (J2)  and (H), the data $$ (\frg, \J, \met)$$  associated with the almost pseudo-Hermitian homogeneous space $$ M = (G/H, \J_{M}, \g_{M})$$ form an \emph{almost $h$-algebra} with kernel $\frh$  (see Definition \ref{def:h_algebra}).
In view of (J4), $(\frg, \J, \met)$ is an $h$-algebra if and only if $M$ is a pseudo-Hermitian manifold. 

%

\medskip

Recall that $\frh = \frg^{\perp}$ is called the kernel of the $h$-structure. 
\begin{definition} \label{def:h_algebra}
An $h$-structure $(\frg, J, \met)$ is called an \emph{almost $h$-algebra} with kernel $\frh$  if 
the following conditions are satisfied:
\begin{enumerate} 
\item $\frh = \ker  J = \frg^{\perp}$ is a subalgebra of $\frg$. 
\item  $\frh$ is contained in $\frg_{h} = \frg_{\met,\J}$.
\end{enumerate}
If $J$ is also  \emph{integrable}, that is,  if 
\begin{enumerate}
\item[(3)] $[J x,J y]-[x,y]- J[x, J y]- J[ J x,y]\in\frh$ for all $x,y\in\frg$, 
\end{enumerate}
then we call  $(\frg, J, \met)$ an \emph{$h$-algebra}.
\end{definition}

\begin{remark}
It follows that (almost) $h$-algebras are precisely the local models of
(almost) pseudo-Hermitian homogeneous spaces, compare Section
\ref{sec:hermitian_models}. 
\end{remark}

\subsubsection{The $h$-algebra models of Hermitian homogeneous spaces of finite volume} 

We note the following application of Theorem \ref{thm:isom_is_quasi-invariant}: 

\begin{cor}\label{cor:hermitian_qi}
Let $M=G/H$ be an almost pseudo-Hermitian homogeneous space with 
finite Riemannian volume.
Then the 
forms $\met$, $\omega$ induced on $\frg$ by $\g_M$, $\omega_M$
are quasi-invariant (with respect to the adjoint representation of $G$).
\end{cor}

This states that the almost $h$-algebra model $(\frg, J, \met)$ for $M$ constitutes a quasi-invariant $h$-structure 	in the sense of Section \ref{sec:nil_h-structure}. 

\subsection{Almost pseudo-Hermitian homogeneous spaces of finite volume} 

We now  derive some global consequences of the algebraic results on $h$-structures and quasi-invariant skew-symmetric forms on Lie algebras as  developed in  
Section \ref{sec:hermitian_nil}. 
The \emph{almost complex isometry group} of an almost pseudo-Hermitian
manifold $(M,\J,\g)$ consists of the diffeomorphisms of $M$ preserving
both $\J$ and $\g$.

{\renewcommand{\themthm}{\ref{mthm:Levi_compact}}
\begin{mthm}
Any almost pseudo-Hermitian homogeneous space $M$ of finite
volume is compact. The maximal semisimple subgroup of the identity component
of the almost complex isometry group of $M$ is compact,
and its solvable radical is nilpotent and at most two-step nilpotent. 
\end{mthm}
}
\begin{proof}
Since the almost complex isometry group $G$ of $M$ acts effectively and
transitively, the almost $h$-algebra model $(\frg, J, \met)$ for $M$ is
quasi-invariant and effective. 

Thus Theorem  \ref{thm:radical_nilpotent} applies and it follows that
$G=KR$, where $K$ is compact semisimple and  the solvable
radical $R$ is nilpotent and at most two-step nilpotent. 
As a consequence of Mostow's result \cite[Theorem 6.2]{mostow4},
the finite volume space $M=(KR)/H$ is in fact compact. 
 \end{proof}

{\renewcommand{\themthm}{\ref{mthm:solvable}}
\begin{mthm}
Let $M$ be a compact almost pseudo-Hermitian homogeneous space with
a transitive effective action by a connected solvable Lie group $G$ of almost complex isometries.
Then $G$ is abelian and compact.
Furthermore, $M$ is a flat complex pseudo-Hermitian torus
$\CC^{p,q}/\,\Gamma$, where $\Gamma$ is a lattice in $\CC^{p,q}$.
\end{mthm}
}
\begin{proof}
As proved in \cite[Theorem A]{BG}, since $G$ is solvable, the stabilizer
subgroup $H$ is a discrete group. Therefore, in the almost $h$-algebra model $(\frg, J, \met)$ 
the scalar product $\met$ on $\frg$ is non-degenerate. 
Then Corollary  \ref{cor:solv_abelian} asserts that $G$ is abelian.
Effectivity of the $G$-action implies that the stabilizer is trivial and that $G$ is compact.

The Lie algebra $\frg$ being abelian also implies that the almost complex structure
on $M$ is integrable, since (J4) is clearly satisfied. Similarly,  any left-invariant pseudo-Riemannian 
connection on an abelian Lie group  is flat.  It follows that $M$ is a flat complex pseudo-Hermitian torus. 
\end{proof}

It is interesting to compare the pseudo-Hermitian case with related situations of considerably 
weaker assumptions. In this regard we mention that the existence of a non-degenerate 
invariant skew-symmetric form on $M$ (e.g.~the fundamental two-form in the pseudo-Hermitian
case) already exhibits a considerable rigidity: 

\begin{cor} \label{cor:almost_sym_ab}
Let $M$ be a  compact manifold equipped with a non-degenerate two-form $\om_{M}$ 
invariant by a locally simply transitive action of a connected solvable Lie group $G$ on $M$.
Then $G$ is abelian and compact and $M$ is a flat symplectic two-torus, that is, 
$\omega$ arises from a parallel symplectic two-form on $\RR^{2n}$. 
\end{cor} 
\begin{proof} The condition that $G$ acts locally simply transitively means that the stabilizer group 
$H$ is discrete, that is,  $M = G/ \Gamma$ is a lattice quotient of $G$. The associated Lie algebra 
model $(\frg, \om)$ for $M$  is therefore given by a quasi-invariant non-degenerate skew-symmetric form on $\frg$.  
By application of Corollary \ref{cor:quasi_solv_om} we infer that $\frg$ is abelian. 
Any left-invariant non-degenerate skew-symmetric two-form on the abelian
group $G$ is closed, hence symplectic.
It is also parallel with respect to the canonical connection.
 \end{proof}


\section{Compact pseudo-Hermitian homogeneous spaces}
\label{sec:complex}

Let $M$ be a compact complex manifold. By a classical result of Bochner and 
Montgomery 
the  group of biholomorphic maps $\mathrm{Hol}(M)$ 
is a complex Lie group and acts holomorphically on $M$.  
We shall consider real connected Lie  subgroups $G$ of $\mathrm{Hol}(M)$ 
and our interest is to study the properties of $G$-invariant geometric structures on $M$. 
The group $G$ itself is not necessarily complex, but its action on $M$ extends to 
that of a complex group $G^{\CC}$ contained in  $\Hol(M)$. If $G$ acts transitively on $M$,
then $G^\CC$ acts transitively on $M$, so that $M=G^\CC/H^\CC$ becomes a complex 
homogeneous manifold. 

\smallskip 

\begin{paragraph}{\em Note on terminology}
By a \emph{complex Lie group} $G^{\CC}$ we mean a 
complex analytic group. A \emph{complex Lie subgroup}  $H^{\CC}$ of $G^{\CC}$ 
is a complex analytic submanifold of $G^{\CC}$ with an induced structure of complex Lie group. 
A  \emph{complex homogeneous space} is a quotient space $G^{\CC}/ H^{\CC}$, 
where $G^{\CC}$ is a complex Lie group and $H^{\CC}$ a closed complex Lie subgroup.  
We let  $M = G^{\CC}/ H^{\CC}$ denote the underlying complex manifold.  
Every complex Lie group $G^{\CC}$ is also a real Lie group in the usual sense, and, if not stated otherwise,  a Lie subgroup of $G^{\CC}$ refers to a  Lie subgroup with respect to the real Lie group structure of $G^{\CC}$. A complex Lie group $A^{\CC}$  isomorphic  to $(\CC^{*})^{k}$ is customarily called a \emph{torus} and we will stick to this convention.  A  \emph{compact abelian complex Lie group}  $T^\CC =\CC^{k}/ \,  \Lambda$ is,  as a real Lie group,  isomorphic to a compact real torus $T^{2k} = (\SS^{1})^{2k}$. The underlying complex manifold of $T^{\CC}$ is called a (compact) \emph{complex torus}.  A complex Lie group is called \emph{reductive} if it is finitely covered by a product of a torus and a semisimple complex Lie group.  In the context of complex Lie groups, any maximal reductive complex subgroup of $G^{\CC}$ is called a Levi subgroup of $G^{\CC}$, whereas in the real case, traditionally,  a Levi subgroup denotes a maximal semisimple Lie subgroup. In the complex Lie group $G^\CC$, any maximal semisimple Lie subgroup $S^\CC$ is a complex subgroup as well.
\end{paragraph}  

\subsection{The Tits fibration}
\label{subsec:titsfibration}

According to Tits \cite{tits}, any compact complex homogeneous manifold $M = G^\CC / \, H^\CC$ fibers holomorphically over a generalized flag
manifold with complex parallelizable fibers.
This means the following: 
The normalizer of $(H^\CC)^\circ$ in $G^\CC$ is
a parabolic subgroup $Q^\CC$, meaning it contains a maximal solvable connected subgroup $B^{\CC}$ 
of $G^\CC$.
Then  $X=G^\CC/ \, Q^\CC$ is a homogeneous complex projective manifold, and it is called 
a \emph{generalized flag manifold}.
In particular, $X$ is compact and simply connected. 
There is an induced holomorphic  fibration $F \to M\to X$ with fiber $F= Q^\CC/H^\CC$. The fiber $F$ is a compact homogeneous complex manifold, and $F =  {F}^\CC/\, \Lambda$ is the quotient of a connected 
complex Lie group ${F}^\CC=Q^\CC/(H^\CC)^\circ$ by the uniform lattice $\Lambda = H^\CC/ (H^\CC)^\circ $
contained  in ${F}^\CC$.
Such a quotient manifold $F$ is called a \emph{complex parallelizable} manifold. The Tits fibration is characterized as the unique holomorphic fibration of the complex manifold $M$ over a flag manifold with parallizable fibers. 

%
\subsubsection{Nilmanifold fiber $F$}

If the complex Lie group ${F}^\CC$ is nilpotent then the fibers of the Tits fibration are compact complex nilmanifolds. 
This assumption has strong consequences for the structure of $M$ and $G^{\CC}$. In the following we assume that $G^\CC \subseteq \Hol(M)$ is a connected complex Lie group that acts effectively and transitively on $M$. Let $N^{\CC}$ denote the nilradical of $G^{\CC}$  and let $S^{\CC}$ be a Levi subgroup.  

\begin{lem} \label{lem:Tits_nilfiber}
Suppose that the Tits fibration of $M$ has nilmanifold fiber $F$. Then:    
\begin{enumerate}
\item $ G^{\CC} = N^{\CC} \cdot S^{\CC}$.
\item  $S^{\CC}$ commutes with $N^{\CC}$. 
\item $N^{\CC} \cap (H^{\CC})^{\circ} = \{1\}$.
\end{enumerate} 
\end{lem}
\begin{proof}
Let $G^{\CC} = R^{\CC}  \cdot S^{\CC}$ be a Levi decomposition, where $R^{\CC}$ is the solvable radical  of $G^{\CC}$.  Then  $P^{\CC} =  S^{\CC} \cap Q^{\CC}$ is parabolic and  contains a  Borel subgroup $B^{\CC}$ of  $S^{\CC}$. Let $A^{\CC}$ denote a maximal torus contained in $B^{\CC}$.
Observe that $Q^{\CC}$ acts holomorphically by complex automorphisms on $F^{\CC}  = Q^{\CC} / (H^{\CC})^{\circ}$. 
Since $A^{\CC}$ acts reductively for every complex representation  and the group $F^{\CC}$ is 
nilpotent,  the image of $A^{\CC}$ in  $F^{\CC}$ must be central.  This means that 
$A^{\CC}$ acts trivially on $F^{\CC}$. As $A^\CC$ is a Cartan subgroup of $B^\CC$, the action of the Borel subgroup $B^{\CC}$ on $F^{\CC}$ is also
trivial. In particular, $B^{\CC}$ acts trivially on the subgroup $R^{\CC}/ (R^{\CC} \cap (H^{\CC})^{\circ})$. 
Since $R^{\CC} \cap (H^{\CC})^{\circ}$ is normalized by $B^{\CC}$, every non-trivial 
subrepresentation of the action of $S^{\CC}$ on the Lie algebra of $R^{\CC}$ is contained in the Lie algebra of  $R^{\CC} \cap (H^{\CC})^{\circ}$,
cf.~Lemma \ref{lem:GV=BV}. It follows that  $R^{\CC} \cap (H^{\CC})^{\circ}$ is normalized by $S^{\CC}$, 
and  $S^{\CC}$ acts trivially on the factor group. 
In particular, $R^{\CC} \cap (H^{\CC})^{\circ}$ is a normal subgroup of $G^{\CC}$. Since 
we assume that $G^{\CC}$ acts effectively on $M$, this implies  $R^{\CC} \cap (H^{\CC})^{\circ} = \{ 1 \}$.
Thus  $R^{\CC}$ embeds as a subgroup into $F^{\CC}$, showing that $R^{\CC} = N^{\CC}$ is nilpotent. 
 \end{proof} 


\begin{lem}\label{lem:GV=BV}
Let $\frs$ be a complex semisimple Lie algebra and $\frb$ a Borel
subalgebra of $\frs$.
If $V$ is a finite-dimensional $\frs$-module, then
$\frs V=\frb V$.
\end{lem}
\begin{proof}
For a certain choice of positive simple roots, let $\frn^+$, $\frn^-$
denote the nilpotent subalgebras spanned by the positive and negative
root spaces, respectively.
Then we may assume $\frb=\fra\ltimes\frn^+$ for some Cartan subalgebra
$\fra$ of $\frs$.
%
Also we  may assume that $V$ is irreducible and non-trivial. Let $\lambda_{-} \neq 0$ be a 
\emph{minimal weight} for the action of $\fra$ on
$V$. Then $\fra V_{\lambda_{-}} = V_{\lambda_{-}}$ and $V =  V_{\lambda_{-}} +  \frn^+ V_{\lambda_{-}}$. 
In particular, $V= \frb V$. 
\end{proof}

\subsubsection{Transitive Lie subgroups of $G^{\CC}$} 

Let 
$G$ be a connected Lie subgroup of $G^\CC$ that acts transitively on $M$. 
Then $ M = G/H $ with $H= G \cap H^\CC$. Moreover,  $L=G\cap Q^\CC$ contains $H$ and  $X=G/L$.
It also follows that $L$ is connected,  since $X$ is simply connected.
The fiber of  $M\to X$ can be identified with $L/H$ 
and therefore the Lie group $L/H^\circ$ is a covering group of
${F}^\CC = L / (L \cap (H^\CC)^\circ) $. 
In particular,  $L/H^\circ$ attains the structure of
a complex Lie group locally isomorphic to  ${F}^\CC$.

\begin{prop} \label{transitive_1} 
Assume that $G^\CC$ admits a transitive Lie subgroup $G$ with a Levi decomposition of the form  $G = N \cdot K$, where $N$ is the nilradical of $G$ and $K$ is compact 
semisimple. Then the Tits fibration of $M$ has nilmanifold fiber. In particular, $K$ centralizes $N$.
\end{prop}

The proof of the proposition builds on the following lemma:  

\begin{lem}  \label{nilpotent}
Let $E^\CC$ be a complex Lie group such that  $E^\CC = V \cdot T$ is an extension of a  compact 
Lie group $T$ by some connected nilpotent normal Lie subgroup $V$ of $E^\CC$. 
Then $E^\CC$ is nilpotent. In particular,  the compact subgroup $T$ is contained in the center of $E^\CC$. 
\end{lem}
\begin{proof} 
By assumption, $V$ is contained in the nilradical $N^\CC$ of $E^\CC$. It also follows by assumption that the Levi subgroup $S^\CC$ of $E^\CC$ is compact. Since $S^\CC$ is a complex semisimple Lie group, this implies that $S^\CC$ is trivial. Hence, $E^\CC$ is solvable and  
$T^\CC_{1} = E^\CC/N^\CC$ is an abelian complex Lie group  and compact. Note that $T_{1}^\CC$ acts by automorphisms on the complex abelian Lie group  $H^{1}(N^\CC) = N^\CC/ \, [N^\CC,N^\CC]$. Since every complex linear representation  of $T^\CC_{1}$ is trivial, $T^\CC_{1}$ acts trivially on  $H^{1}(N^\CC)$. Therefore, the compact group $T$ acts trivially on  $H^{1}(N^\CC)$ and (thus) also on $N^\CC$. Hence,  $E^\CC = N^\CC T$ is nilpotent. 
\end{proof} 

\begin{proof}[Proof of Proposition \ref{transitive_1}]
First note that $N$ is contained in $Q^\CC$, since every solvable Lie group of complex automorphisms of the projective variety $X$ has a fixed point. Let $L = G \cap Q^{\CC} = N \cdot C$, where $C = K \cap Q^{\CC}$ is compact. By the above remarks,   $L/H^\circ$ is a complex Lie group covering $F^\CC$ and it satisfies the assumption of Lemma~\ref{nilpotent}. Hence,  $L/H^\circ$ is nilpotent. So ${F}^\CC=Q^\CC/(H^\CC)^\circ$  is nilpotent, which means that the Tits fibration for $M$ has nilmanifold fiber.  By Lemma \ref{lem:Tits_nilfiber}, $G^\CC = N^\CC S^\CC$, where $S^\CC$ centralizes $N^\CC$. We may assume that $N$ is contained in $N^\CC$ and $K$ in $S^\CC$. It follows that $K$ centralizes $N$.
\end{proof} 

\begin{remark}
Since $N$ is contained in $L$, $X = K/C$ for the compact subgroup
$C = L \cap  K$. 
Since the parabolic subgroup $P^\CC= Q^\CC \cap S^\CC$ of $S^\CC$ is the centralizer of a torus \cite[4.15 Th\'eor\`eme]{BT},  $C= K \cap P^\CC$ is the centralizer of a torus $T_0$ in $K$.
In particular, $C$ contains a maximal torus $T$ of $K$. We may write  $C = T_{0} C_{0}$, where $C_{0}$ is compact semisimple. 
Since $L/H^{\circ}$ is nilpotent, $C_{0}$ is contained in  $H^{\circ}$.  Thus  $H^{\circ} \cap K = T_{1}  C_{0}$, where $T_{1}$ is contained in $T_{0}$. Choose a compact torus $T_{2}$ such that $T_{0} = T_{1} T_{2}$. 
Then $L/H^{\circ} = N \, \bar T_{2}$, where  $\bar T_{2}$ is a compact torus that centralizes $N$. 
\end{remark}

We infer the following addendum to Lemma \ref{lem:Tits_nilfiber}:

\begin{lem} \label{lem:Titsnilfiber2}
Suppose that $M$ has nilmanifold fiber. Then $F^\CC/ N^\CC$ is compact and $N^\CC$ has compact center. 
\end{lem}
\begin{proof}  Let $K$ be a compact real form of $S^\CC$. Since $K$ is a maximal compact subgroup,
$S^\CC = K \cdot B^\CC$ by the Iwasawa decomposition.
Hence, $K$ acts transitively on $X$. Thus the Lie subgroup  $G = N^{\CC} K$ acts transitively on $M$. As the above remark shows, $F^\CC = N^\CC \bar T_2$, where $\bar T_2$ is a compact torus. Since $F = F^\CC / \Lambda$, where $\Lambda = H^\CC/(H^\CC)^\circ $, is compact,  $\Lambda$ intersects $N^\CC$ in a uniform lattice. Since $N^\CC$ is nilpotent, it follows that $\Lambda$ also intersects the center $Z$ of $N^\CC$ in a uniform lattice. However, since $Z$ is also central in $G^\CC$ by Proposition \ref{transitive_1}, $H^\CC \cap Z$ is trivial. We deduce that $Z$ is compact.
\end{proof}

%
%
\subsubsection{$G$-invariant pseudo-Hermitian metric} 

We say that the Tits fibration for $M$ has \emph{torus fiber} if $F^{\CC}$ is abelian. 
This means that the fiber  $F = F^{\CC}/\, \Lambda$ of the Tits fibration is a compact abelian complex Lie group. In fact, as we assume that the action
of $G^\CC$ on $M$ is effective, $F^\CC$ is a compact complex Lie group due to Lemma \ref{lem:Titsnilfiber2} above. It follows: 

\begin{lem} \label{lem:Tits_torusfiber}
If the compact complex homogeneous space $G^\CC/ H^\CC$ has torus fiber then the solvable radical of $G^\CC$ is a compact complex abelian group.
\end{lem}
%

The following theorem characterizes compact complex homogeneous manifolds with an underlying homogeneous pseudo-Hermitian structure:

{\renewcommand{\themthm}{\ref{mthm:torus_fiber}}
\begin{mthm}
Let $M$ be a compact complex homogeneous manifold, and let $G$ be a
closed Lie subgroup of $\Hol(M)$.
Assume that $G$ preserves a
pseudo-Hermitian metric and acts transitively on $M$.
Then the Tits fibration of $M$ has torus fiber and the group $G$ is compact.
%
%
\end{mthm}
}
\begin{proof}
Since $G$ acts holomorphically by isometries of a pseudo-Hermitian metric on $M$, it also preserves its associated fundamental 
two-form $\om$ on $M$. By Theorem \ref{thm:radical_nilpotent}, the existence of the invariant non-degenerate two-form $\om$  implies $G = N \cdot K$, where $N$ is a connected nilpotent  normal subgroup and $K$ is compact semisimple. Since $G$ is a subgroup of $G^{\CC}$ that acts transitively on $M$, Proposition \ref{transitive_1} asserts that $K$ centralizes $N$. Theorem  \ref{thm:radical_nilpotent} also states that the commutator subgroup $[N,N]$ is contained in the center of $N$. Thus,  $[N,N]$  is contained in the center of $G$. In view of the remark following Theorem \ref{thm:radical_nilpotent}
, this implies that $[N, N]$ is contained in $H^{\circ}$, showing that the normal subgroup $[N,N]$ of $G$ is
trivial, since $G$ acts effectively. Therefore, $N$ is abelian. By the remark following the proof of Proposition \ref{transitive_1}, $L/H^{\circ} = N \bar T_{2}$ is abelian. Therefore, 
$F^{\CC}$ is abelian and the Tits fibration of $M$ has torus fiber. 
Lemma \ref{lem:Tits_torusfiber} implies that  $N^\CC K$ is compact. Hence, the closed subgroup $G$ is compact.
\end{proof}

\begin{cor} \label{cor:torus_fiber} 
A compact complex homogeneous manifold $M$ admits a homogeneous pseudo-Hermitian structure if and only if its Tits fibration has torus fiber.
Moreover, every such manifold admits a homogeneous 
positive definite Hermitian metric. 
\end{cor}
\begin{proof} If the Tits fibration of $M = G^\CC/ H^\CC$ has torus fiber we may always choose an underlying homogeneous Hermitian structure on $M$: In fact, by Lemma \ref{lem:Tits_torusfiber}, we have  $G^\CC = T^\CC S^\CC$, where $T^{\CC}$ is compact and central in $G^\CC$ and  $S^\CC$ is a semisimple complex Lie group. Let $K$ be a compact real form of $S^\CC$. As noted in the proof of Lemma \ref{lem:Titsnilfiber2}, $G = T^{\CC} K$ acts transitively on $M$. Invariant integration on $M = G/H$ furnishes a $G$-invariant (definite) Hermitian structure.
\end{proof}

\begin{remark} Hano and Kobayashi studied the implications of the existence of $G$-invariant volume forms on complex homogeneous spaces. The assertion of Theorem \ref{mthm:torus_fiber} for  a definite Hermitian metric is actually contained in \cite[Theorem B]{HanoKob}. 
\end{remark}

\subsection{Applications of Theorem \ref{mthm:torus_fiber}}

Let $(M,\g,\J)$ be a compact homogeneous pseudo-Hermitian manifold.
Taking $G = \Aut(M,\g,\J)^\circ$ in Theorem \ref{mthm:torus_fiber}
 we obtain that
the identity component of the holomorphic isometry group of a
compact homogeneous pseudo-Hermitian manifold is compact.
In the simply connected case, this can be strengthened 
to: 
{\renewcommand{\themthm}{\ref{mcor:hermitian}}
\begin{mcor}
The holomorphic isometry group of a simply connected
compact homogeneous pseudo-Hermitian manifold is compact.
\end{mcor}
}
\begin{proof}
Since $M$ is simply connected, a maximal compact semisimple subgroup
$K$ of $G$ already acts transitively on
$M$ (\cf \cite{montgomery}).  
Therefore we may identify $\Tan_x M$ at $x\in M$ with $\frk/(\frh\cap \frk)$,
where $\frk$ is the Lie algebra of $K$. 
Now $G =KA$ is an almost direct product of a compact semisimple Lie
group $K$ and a compact abelian group $A$. In particular, 
 $K$ is characteristic in $G$ and the isotropy representation of the stabilizer $\Aut(M,\g,\J)_x$  
factorizes over a closed subgroup of the automorphism group of $K$.
As this latter group is compact, the isotropy representation has compact
closure in $\GL(\Tan_x M)$. If follows that there exists a Riemannian
metric on $M$ that is preserved by $\Aut(M,\g,\J)$. 
Hence $\Aut(M,\g,\J)$ is compact.
\end{proof}

A particular interesting case for Corollary \ref{mcor:hermitian} arises if in the Tits fibration 
the fiber is trivial, that is, if $M$ is a generalized flag manifold. In fact, generalized flag 
manifolds are always simply connected and have non-zero Euler characteristic, conversely 
any compact complex homogeneous  space with non-zero Euler characteristic is a generalized flag manifold 
(see \cite{goto,tits}). 

\begin{example}\label{ex:euler_non0}
Let $M$ be a compact homogeneous pseudo-Hermitian manifold with
non-zero Euler charac\-teristic. Then $M$ is biholomorphic to a generalized 
flag manifold. By Corollary \ref{mcor:hermitian}, 
the holomorphic isometry group $G$ of $M$ is compact.
Since the Euler characteristic of $M$ is non-zero, the stabilizer $H$
contains a maximal torus of $G$.
Hence $G=K$ is a compact semisimple Lie group.
Moreover, $M$ is a product of pseudo-Hermitian irreducible flag manifolds.
\end{example} 
\begin{proof}
The only thing that remains to be proved is the last statement.
Let $\frg_1$, $\frg_2$ be two simple factors of $\frg$, and let
$x_1\in\frg_1$, $x_2\in\frg_2$.
Since $\frh$ contains a maximal torus $\frt$ of $\frg$,
$x_1$ can be written $x_1=[t,y_1]$ for some $t\in\frt\cap\frg_1$
and $y_1\in\frg_1$.
Recalling that the induced symmetric bilinear form $\met$ on $\frg$ is
$\frh$-invariant, we obtain
\[
\langle x_1, x_2\rangle
=
\langle [t,y_1],y_2\rangle
=
-\langle y_1, [t,y_2]\rangle
= 0.
\]
It follows that the simple factors of $\frg$ are mutually orthogonal
with respect to $\met$. This implies that the pseudo-Hermitian metric
$\g$ on $M$ is a product of the metrics induced on the irreducible factors of $M$.
\end{proof} 

%
%
%

%
\subsubsection{Examples}

In general, the decomposition of $\Aut(M,\g,\J)^\circ$ into a compact semisimple
factor and an abelian factor does not translate into a splitting of $M$
into corresponding complex factors.
This distinguishes the case of arbitrary pseudo-Hermitian metrics
from that of pseudo-K\"ahler metrics (discussed in  Appendix A).

\begin{example}[Complex structures on the unitary group $\UU(2)$]\label{ex:no_product}
Let $\Lambda$ be a discrete uniform subgroup in $\RR^{>0} \leq \CC^*$ and let 
$$    M =   (\CC^2 - \zsp)  /\,    \Lambda    $$
be  a  two-dimensional Hopf manifold, 
which  is a homogeneous space for $\GL(2,\CC)$.
The complex homogeneous manifold
$$
\tilde M = \GL(2 ,\CC) / \Aff(\CC)
= \CC^2-\zsp
$$
is a covering manifold of $M$,
where $\Aff(\CC) = \{ \left(\begin{smallmatrix}
	\alpha & u \\ 0 & 1
\end{smallmatrix}\right)\ \mid\ \alpha \neq 0 \}$.
Observe that  $\tilde M$ has a holomorphic fibration 
\begin{equation*} \label{eq:holf}
\CC^*  \longto \tilde M  \longto   \mathrm{P}^1 \CC = \SL(2,\CC) / B \, ,
\end{equation*}
where $B$ is the Borel subgroup of $\SL(2, \CC)$ and $\CC^*$ acts as the center of $\GL(2,\CC)$.  
We then put 
$$  G^\CC = \GL(2,\CC) / \Lambda \quad  \text{ and } \quad
M= G^\CC / \Aff(\CC) . $$
Note that $ G^\CC$ has maximal compact subgroup $ \SU(2) \cdot  T^1_{\CC}$, where $T^1_\CC = \CC^*/\Lambda$.
Let $N = \RR^{>0}/ \Lambda$ and observe that the real compact Lie subgroup $$G = \SU(2) \cdot N $$ acts simply transitively on $M$.  Since $G$ is isomorphic to $\UU(2)$,  $M$ is biholomorphic to the compact Lie group 
 $\UU(2)$ endowed with a left-invariant complex structure. Clearly, 
$M$ (being  homeomorphic to $\SS^1 \times \SS^3$)  
cannot be K\"ahler, and neither is it a product of complex manifolds.
Consider the  compact torus  $\SS^{1} = \SU(2) \cap B$. Then  
the Tits fibration of $M$ (induced by $T^1_{\CC}$) takes the form
$$  N \SS^{1} \longto M = G  \longto   \mathrm{P}^1 \CC = \SS^{3}/\, \SS^{1} \; .  $$
\end{example}

\begin{remark}
Left-invariant pseudo-Hermitian metrics on  the unitary group $\UU(2)$ are computed in \cite[Theorem 4.6]{ACHK}).
These  metrics  are \emph{locally conformally pseudo-K\"ahler}.
In fact, \cite[Theorem 1]{HK} shows that the Tits fibration of a
complex compact homogeneous manifold with a (non-K\"ahler) locally
conformally K\"ahler metric has one-dimensional fiber.
\end{remark}

The Hopf manifolds are not simply connected. It is a famous observation due to Calabi and Eckmann \cite{CE} that there are non-K\"ahler simply connected compact homogeneous complex manifolds. 

\begin{example}[Calabi-Eckmann manifolds]
Let $\Lambda \cong \CC$ act on  $\CC^{p}  \times \CC^{q} $ via $(u, v) \mapsto ( \e^{t} u, \e^{\alpha t} v)$, where $\alpha \in \CC - \RR$, and consider
$$    M  = \bigl((\CC^{p} - \zsp)\times(\CC^{q}  - \zsp)\bigr)/  \Lambda  \;  , \quad  \text{ $p,q >1$} .   $$ 
Then $M$ is a compact complex homogeneous manifold for $\GL(p,\CC) \times \GL(q, \CC)$. Its 
Tits fibration is of the form 
$$    \CC/ (\ZZ + \alpha \ZZ)   \longto   M  \longto  \mathrm{P}^{p-1} \CC\times  \mathrm{P}^{q-1} \CC $$
and has torus fiber (an elliptic curve). In particular, $M$ admits invariant (pseudo-) Hermitian metrics 
and is diffeomorphic to $\SS^{2p-1} \times \SS^{2q-1}$. 

\end{example}

\appendix


\section{Additional proofs}
\label{sec:newproof}

In this appendix we derive 
that any compact homogeneous 
pseudo-K\"ahler manifold is a  pseudo-K\"ahlerian product of homogeneous spaces for a compact 
semisimple group and  for a compact abelian group.
Besides the compactness of the identity component of the holomorphic
isometry group, this 
is the essential content of Dorfmeister and Guan's theorem \cite{DG}.
The splitting is distinctive for the K\"ahler and symplectic case. 
For the proof  we use our Theorem \ref{mthm:torus_fiber} on compact pseudo-Hermitian homogeneous spaces and specialize it to the case  that the fundamental  two-form is closed.
Note that proofs for Dorfmeister and Guan's theorem using the properties of
Hamiltonian group actions have been given by Huckleberry \cite{huckleberry}
and Guan \cite{guan}.

%
%
%
\subsection{Splitting of compact homogeneous symplectic manifolds} \label{A1}

We  first review the important splitting  result on homogeneous symplectic manifolds
given by Zwart and Boothby \cite[Theorem 5.11, Corollary 6.10]{ZB}:  

\begin{thm}[Zwart \& Boothby]\label{thm:ZB}
Let $M$ be a compact symplectic homogeneous manifold, with
$G\subseteq\Aut(M,\omega_M)$ a connected Lie group acting transitively
on $M$, so that $M=G/H$ for some closed subgroup $H$ of $G$. Then:
\begin{enumerate}
\item
$G$ is a  direct product $G=KR$, where $K$ is a compact semisimple
subgroup with trivial center and $R$ is the solvable radical of $G$.
\item
$H$ is a  direct product of $H_K=H\cap K$ and $H_R=H\cap R$.
Moreover, $H_K$ is connected and is the centralizer of a torus. 
\item
$M=M_K\times M_R$ splits as a product of compact symplectic homogeneous
spaces $M_K=K/H_K$ and $M_R=R/H_R$.
\item
$R=A\ltimes N$ with an abelian subgroup $A$ and abelian nilradical $N$,
where $A$ acts on $N$ by rotations.
\end{enumerate}
\end{thm}

\subsubsection{Nil-invariant closed skew-symmetric forms}

Let $\om$  be a nil-invariant closed skew-symmetric form on the Lie algebra
$\frg$. A skew-symmetric form on $\frg$ is called \emph{closed},  if for all $x,y,z\in\frg$,
\begin{equation}
\omega([x,y],z) + \omega([y,z],x) + \omega([z,x],y) = 0.
\label{eq:closed}
\end{equation}
Since $\om$ is nil-invariant, we may assume that $\frg = \frk \ltimes \frr$, where $\frk$ is a compact semisimple Lie subalgebra and $\frr$ is the solvable radical of $\frg$. Furthermore, let $\frn \subseteq \frr$ denote the nilradical of $\frg$. The following two results can then be seen as an addendum to 
Section \ref{sec:nil_skew}:

\begin{lem}[Nil-invariant closed skew-symmetric forms] \label{lem:symplectic}
\hspace{1cm}
\begin{enumerate}
\item $[\frg, \frg] \, \perpo \, \frn$ and $\frk  \, \perpo \, \frr$. 
\item $[\frk, \frn] \,  \perpo \, \frg$ and $[\frn, \frn] \,  \perpo \, \frg$.
\item $\frg^{\perpo} = (\frg^{\perpo} \cap \frk) +  (\frg^{\perpo} \cap \frr )$. 
\end{enumerate}
\end{lem}
\begin{proof}
Using \eqref{eq:closed} and $\frn \subseteq  \frg_\omega$, 
for all $x\in \frn$,  $k_{1}, k_{2} \in \frg$,  we compute 
\begin{align*}
 \omega([k_1,k_2],x)
&=-\omega([k_2,x],k_1)-\omega([x,k_1],k_2) 
=\omega(k_2,[k_1,x])-\omega([x,k_1],k_2)
=0.
\end{align*}
Hence,  $[\frg, \frg] \perp_{\omega}\frn$. Recall that $\frk = [\frk, \frk]$ and $[\frk, \frr] \subseteq \frn$. 
Thus,  for any $k = [k_1,k_2]$, and $r \in \frr$,  $ \omega(k, r) =  -\omega([k_2,r],k_1)-\omega([r,k_1],k_2) = 0$.
Hence, $\frk \perpo \frr$. Finally, $\omega([k,x],  r)  =  -\omega([x,r],k)-\omega([r,k],x) = 0$, for $x \in \frn$. 
This shows $[\frk, \frn] \perpo \frr$. Since  $[\frk, \frn] \subseteq \frn \perpo \frk$, this implies 
$[\frk, \frn] \perpo \frg$. Similarly, $[\frn, \frn] \perpo \frg$. Hence (1) and (2) hold and (3) follows. 
\end{proof}

\begin{prop}  \label{pro:symplectic}
Suppose that $(\frg, \om)$ is effective. Then:
\begin{enumerate}
\item $\frn$ is abelian.  
 \item $\frg = \frk \times \frr$ is a direct product. 
\end{enumerate}
 \end{prop}
 \begin{proof} By the previous lemma, $[\frn, \frn] \subseteq \frg^{\perpo}$. Part (1) follows.
 Since $\frn$ is abelian, $[\frk, \frn]$ is an ideal in $\frg$ and contained in $\frg^{\perpo}$. Thus $[\frk, \frn] = \zsp$.
 \end{proof}
 
 \subsubsection{Proof of the symplectic splitting theorem}  \label{sec:symplectic_split}
 In the algebraic model $(\frg, \om)$ for  the homogeneous symplectic manifold  
 $\om$ is a quasi-invariant \emph{closed} skew-symmetric form and  $\frg^\perpo = \frh$ is  the Lie algebra of 
 $H$. Moreover,  $(\frg, \om)$ is effective.  

According to  Proposition \ref{prop:Sorthogonal},  
$\frg = \frk \ltimes \frr$, where $\frk$  compact semisimple 
and $\frr$ is  solvable. By Proposition \ref{pro:symplectic},  $\frk$ and $\frr$ commute. Hence $G = K \cdot R$ is an almost direct product, meaning that both $K$ and $R$ are closed normal subgroups  of $G$, 
 and that  their intersection is finite and central in $G$. 
 
Moreover,  $\Ad(k)|_{\frr}=\id_{\frr}$ and
$\Ad(r)|_{\frk}=\id_{\frk}$ for all $k\in K$, $r\in R$. Let $G_{\om}$ denote the 
subgroup formed by all elements  $g \in G$ such that $\Ad(g) \in \OO(\om)$.  
Since $\frk \perpo \frr$, by  Lemma \ref{lem:symplectic},  it follows that
$g\in G_{\omega}$ if and only if $g=kr$ with $k\in K\cap G_{\omega}$ and
$r\in R\cap G_{\omega}$.  
Furthermore, by construction,  we have that  $H \subseteq G_{\om}$. 

By Borel \cite[Theorem 1]{borel_kaehlerian}, 
$K$ and  $K_{\om} = K\cap G_{\omega} $ are  connected, $K_{\om}$ is the centralizer of a torus, 
and $ K_{\omega} \subseteq H$. Therefore $H = (H \cap K) (H \cap R)$ (see  \cite[5.10, 5,11]{ZB}). 
Note that $K \cap R$ is central in $K$ and therefore trivial.  
In the view of  Lemma \ref{lem:symplectic} (3),  we have now proved (1) to (3) of Theorem \ref{thm:ZB}.

\subsection{Compact homogeneous pseudo-K\"ahler manifolds} \label{A2}

Let $M$ be a compact homogeneous pseudo-K\"ahler manifold.
That is, $M$ is pseudo-Hermitian manifold where the metric
$\g_M$ is a pseudo-K\"ahler metric on $M$, which means that the fundamental two-form 
$\omega_M$ is closed. 

\begin{thm}[Dorfmeister \& Guan]\label{thm:DG}
Let $M$ be a compact homogeneous pseudo-K\"ahler manifold,
$G=\Aut(M,\g_M,\J_M)^\circ$, and $H$ a closed subgroup of $G$ such that $M=G/H$.
Then:
\begin{enumerate}
\item
$G$ is compact and $H\subseteq K$.
\item
$M=T\times M_K$, where $T$ is a complex torus and
$M_K=K/H$ is a rational homogeneous variety.
\item
$H$ is connected, has trivial center and is the centralizer of a torus
in $K$.
\item
$M_K=M_{K_1}\times\cdots\times M_{K_m}$ is a product of pseudo-K\"ahler
manifolds, where the $K_i$ are the simple factors of $K$ and
$M_{K_i}=K_i/(H\cap K_i)$.
\end{enumerate}
\end{thm}

From Theorem \ref{mthm:torus_fiber} we know that the Lie algebra $\frg$ of
$G$ is a product $\frk\times\frr$, where $\frk$ is semisimple of compact
type and the solvable radical $\frr$ is abelian.
Recall that $\frg^\perp=\frg^\perpo = \frh$ is the Lie algebra of $H$,
where $\omega$ is the fundamental two-form of the effective $h$-structure 
$(\frg, J, \met)$,  which describes the homogeneous pseudo-K\"ahler manifold, 
see Section \ref{sec:hermitian_models}. 

\begin{lem}\label{lem:pk_kr_orthog}
$\frk\perp_{\omega}\frr$ and $\frg^\perp=\frk^{\perp_{\omega}}\cap\frk$.
\end{lem}
\begin{proof} This is a direct consequence of Lemma \ref{lem:symplectic}, 
taking into account that, since $\frr$ is the center of $\frg$, we have $\frg^\perp\cap\frr=\zsp$. 
\end{proof}

Let $\frk_1,\ldots,\frk_m$ denote the simple ideals of $\frk$
and $\frh_i=\frg^\perp\cap\frk_i$ for $i=1,\ldots,m$.

\begin{lem}\label{lem:pk_kr_orthog2}
$\frk\perp_{\met}\frr$, and $\frk_i\perp_{\met}\frk_j$ for all $i\neq j$.
In particular,
$\frg^\perp=\frh_1\times\cdots\times\frh_m$.
\end{lem}
\begin{proof}
Since $\frk$ is semisimple and $\omega$ is closed, there exists a linear
form $\lambda\in\frk^*$ such that $\omega(x,y)=\lambda([x,y])$ for all
$x,y\in\frk$, and since the Killing form $\kappa$ of $\frk$ is
non-degenerate, there exists $a\in\frk$ such that $\lambda=\kappa(a,\cdot)$.
By invariance of $\kappa$, $\frk^{\perp_{\omega}}\cap\frk=\zen_{\frk}(a)$.
Combined with Lemma \ref{lem:pk_kr_orthog},
\[
\frg^\perp=\frk^{\perp_{\omega}}\cap\frk=\zen_{\frk}(a).
\]
As $a\in\frk$ is contained in some maximal torus of $\frk$ which in
turn is contained in $\zen_{\frk}(a)$, it follows that $\frg^\perp$
is its own normalizer in $\frk$.
By (J3) (\cf Section \ref{sec:hermitian_models}), $[\J\frr,\frg^\perp]\subseteq\frg^\perp$, so $\J\frr$ normalizes
$\frg^\perp$ and thus $\J\frr\subseteq\frr+\frg^\perp$.
It follows that for all $r\in\frr$, $x\in\frk$,
\[
\langle x,r\rangle = \omega(x,\J r) = 0,
\]
which proves $\frk\perp_{\met}\frr$.

Suppose $x_i\in\frk_i$, $y_j\in\frk_j$ for $i\neq j$. We may assume that
$x_i=[x_i',x_i'']$ for some $x_i',x_i''\in\frk_i$. By \eqref{eq:closed},
\[
\omega(x_i,y_j)
=
\omega([x_i',x_i''],y_j)
=
-\omega([x_i'',y_j],x_i') -\omega([y_j,x_i'],x_i'')
= 0.
\]
It follows that $\frk_i\perp_{\omega}\frk_j$.
This immediately implies
\[
\frg^\perp = \frh_1\times\cdots\times\frh_m.
\]
Similar to above we can now use (J3) to find that $\frk_j\cap\J\frk_i$ normalizes
$\frg^\perp$ and is thus contained in $\frk_j\cap(\frr+\frg^\perp)=\frh_j$.
So
\[
\langle x_i,y_j\rangle = \omega(\J x_i,y_j) = 0,
\]
which concludes the proof.
\end{proof}


Now we prove Dorfmeister and Guan's theorem.

\begin{proof}[Proof of Theorem \ref{thm:DG}]
As remarked in \ref{sec:symplectic_split}, $ H = (K\cap G_{\omega})  (H \cap R)$. 
Since $G$ is compact by Theorem \ref{mthm:torus_fiber},  $H\cap R$ is central 
in $G$ and hence trivial, so  $H\subseteq K$. 

By Lemmas \ref{lem:pk_kr_orthog} and \ref{lem:pk_kr_orthog2},
$\frk\perp\frr$ with respect to both $\omega$ and $\met$.
Hence $\J\frk\subseteq\frk$ and $\J\frr\subseteq\frr+\frg^\perp$.
We may thus assume that $\J\frr\subseteq\frr$ without changing the complex
structure on $M$. It follows that $M_K=K/H$ and $R$ are both compact
pseudo-K\"ahler manifolds, where $R$ is a compact torus.

The claim  (3) and the symplectic splitting in (4) is the content of the aforementioned theorem of Borel \cite[Theorem 1]{borel_kaehlerian}. 
Moroever, $M_K=M_{K_1}\times\cdots\times M_{K_m}$ is a
product of pseudo-K\"ahler manifolds 
by Lemma \ref{lem:pk_kr_orthog2}.
\end{proof}

%


\subsection{Metric compact symplectic homogeneous spaces}
\label{sec:noncompatible}

In this section, we adopt a slightly different point of view and
study compact homogeneous manifolds $M=G/H$ equipped with an invariant
symplectic form $\omega_M$ and an invariant pseudo-Riemannian metric
$\g_M$, but with no assumption of compatibility between the two.

\begin{thm}\label{thm:metric_symplectic}
Let $M$ be a compact symplectic homogeneous manifold, with
$G\subseteq\Aut(M,\omega_M)$ a connected Lie group acting transitively
on $M$, so that $M=G/H$ for some closed subgroup $H$ of $G$.
In addition, suppose $M$ is equipped with a $G$-invariant pseudo-Riemannian
metric $\g_M$. Then:
\begin{enumerate}
\item
$G$ is an almost direct product $G=KR$, where $K$ is a compact semisimple
subgroup and the solvable radical $R$ of $G$ is a compact abelian subgroup.
\item
$H\subseteq K$ is connected and contains the center of $K$.
\item
$M=(K/H)\times R$ is a product of symplectic manifolds.
\end{enumerate}
\end{thm}
\begin{proof}
The metric $\g_M$ induces a nil-invariant symmetric bilinear form $\met$ on
$\frg$, and,
as in \eqref{eq:GperpH}, $\frg^{\perp_{\omega}}=\frg^{\perp_{\met}}$ is the
Lie algebra of $H$, and we simply write $\frg^\perp$ for it.

The subalgebra $\frg^\perp$ contains no non-trivial ideal of $\frg$, and
by the splitting in Theorem \ref{thm:ZB}, also
\[
\frk^{\perp_{\omega}}\cap\frk=\frg^\perp\cap\frk
\]
contains no non-trivial ideal of $\frk$ and
\[
\frr^{\perp_{\omega}}\cap\frr=\frg^\perp\cap\frr
\]
contains no non-trivial ideal of $\frr$.

By Theorem \ref{thm:BGZ_invariance}, the restriction of $\met$ to $\frr$
is an invariant bilinear form. Hence $\frg^{\perp_{\met}}\cap\frr$ is an
ideal in $\frr$, and since $[\frk,\frr]=\zsp$, it is an ideal in
$\frg$. Hence
\[
\frr^{\perp_{\omega}}\cap\frr=\frg^\perp\cap\frr=\frg^{\perp_{\met}}\cap\frr=\zsp.
\]
This means $\omega$ is non-degenerate on $\frr$.
From the splitting $\frg^\perp=(\frg^\perp\cap\frk)\times(\frg^\perp\cap\frr)$
it follows that $\frg^\perp\subseteq\frk$.
By Corollary  \ref{cor:almost_sym_ab}, $R$ is abelian.
Then $H_R$ is a normal subgroup of $G$, and since $G$ acts effectively,
$H_R$ is trivial. So $R$ is compact.
\end{proof}

\begin{example}\label{ex:non_splitting}
Let $M_K=K/H$ be any symplectic homogeneous space for a compact
semisimple Lie group $K$ and a suitable closed subgroup $H$.
Consider some vector space decomposition $\frk=\frh\oplus W$ of the Lie
algebra $\frk$ of $K$.
Let $d=\dim W$ and equip $\RR^d$ with the canonical symplectic
form.
On the Lie algebra $\frk\times\RR^d$, define a symmetric bilinear form
$\met$ by $\met|_{\frk\times\frk}=0$, $\met|_{\RR^d\times\RR^d}=0$ and
such that $W$ and $\RR^d$ are dually paired.
This $\met$ extends to a symmetric bilinear form on the
Lie group $G=K\times\TT^d$, and $\met$ is trivially nil-invariant.
Then $M=G/H$ is a compact symplectic homogeneous space with a
pseudo-Riemannian metric. Furthermore, $M$ splits as a product of
symplectic manifolds $K$ and $\TT^d$, but not as a product of
pseudo-Riemannian manifolds.
\end{example}




\end{document}